\documentclass[11pt, reqno]{amsart}
\usepackage{mathrsfs}
\usepackage[active]{srcltx}
\usepackage{mathrsfs}
\usepackage{mathtools}
\usepackage{esint}
\usepackage{longtable}
\usepackage{bbm}
\usepackage{xcolor}

\usepackage{xcolor}
\definecolor{bluecite}{HTML}{0875b7}
\usepackage[unicode=true,
bookmarksopen={true},
pdffitwindow=true,
colorlinks=true,
linkcolor=bluecite,
citecolor=bluecite,
urlcolor=bluecite,
hyperfootnotes=false,
pdfstartview={FitH},
pdfpagemode= UseNone]{hyperref}

\newtheorem{proposition}{Proposition}[section]
\newtheorem{theorem}{Theorem}[section]

\newtheorem{remark}{Remark}[section]
\numberwithin{equation}{section}

\usepackage{bbm}  
\usepackage{dsfont}

\address{\textsc{Zolt\'an M. Balogh}: 
	Mathematisches Institut, University of Bern, Sidlerstrasse 12, 3012, Bern, Switzerland.}
	\email{zoltan.balogh@math.unibe.ch}
\address{\textsc{Sebastiano Don}:
Mathematisches Institut, University of Bern, Sidlerstrasse 12, 3012, Bern, Switzerland.}
\email{sebastiano.don@math.unibe.ch}
\address{\textsc{Alexandru Krist\'aly}: 
Department of Economics, Babe\c s-Bolyai University, Cluj-Napoca, Romania  \& Institute of Applied Mathematics, \'Obuda
	University, 
	Budapest, Hungary. 
}
\email{alex.kristaly@econ.ubbcluj.ro; kristaly.alexandru@uni-obuda.hu
}

\thanks{Z. M. Balogh and S. Don were
	supported by the Swiss National Science Foundation, Grant Nr. {200020\_191978}. A. Krist\'aly  was supported by the Excellence Scholarship
	Program \'OE-KP-2-2022 (Hungary) and by the  UEFISCDI/CNCS grant PN-III-P4-ID-PCE2020-1001 (Romania).}

\title[Sharp log-Sobolev inequalities]{Sharp weighted  log-Sobolev inequalities: characterization of equality cases and applications}

\date{\today}

\author[Z.\ M.\ Balogh, S.\ Don and A.\ Krist\'aly]{Zolt\'an M. Balogh, Sebastiano Don and Alexandru Krist\'aly}

\begin{document}
	\begin{abstract}
By using optimal mass transport theory, we provide a direct proof to the sharp $L^p$-log-Sobolev inequality $(p\geq 1)$ involving a log-concave homogeneous weight on an open convex cone $E\subseteq \mathbb R^n$. The perk of this proof is that it allows to characterize the extremal functions realizing the equality cases in the $L^p$-log-Sobolev inequality. The characterization of the equality cases is new for  $p\geq n$ even in the unweighted setting and $E=\mathbb R^n$. As an application, we provide a sharp weighted hypercontractivity estimate for the Hopf-Lax semigroup related to the Hamilton-Jacobi equation, characterizing also the equality cases.  
	\end{abstract}
	\maketitle 

\vspace{-0.7cm}
\section{Introduction}
The first version of the optimal Euclidean $L^p$-log-Sobolev inequality for $1 < p < n$ 
    goes back to Del Pino and Dolbeault \cite{delPinoDolbeault-2}, whose proof is based on the Gagliardo-Nirenberg inequality. It states that for Sobolev functions $u \in {W}^{1,p}(\mathbb R^n)$ with $\displaystyle\int_{\mathbb R^n} |u|^p dx = 1$ one has the inequality 
\begin{equation}\label{e-sharp-log-Sobolev}
\int_{\mathbb R^n}  |u|^p\log |u|^p dx\leq \frac{n}{p}	\log\left(\mathcal L_{p}\displaystyle\int_{\mathbb R^n} |\nabla u|^p dx\right);
\end{equation}
here, the sharp constant has the form
$$\mathcal L_{p} = \frac{p}{n}\left(\frac{p-1}{e}\right)^{p-1}\left(\Gamma\left(\frac{n}{p'}+1\right)\omega_n\right)^{-\frac{p}{n}},$$
where $p'= p/(p-1)$ and $\omega_n$ is the volume of the unit ball in $\mathbb R^n$. Furthermore, in \cite{delPinoDolbeault-2} it was shown by a symmetrization technique that equality in \eqref{e-sharp-log-Sobolev} holds if and only if the extremal functions $u\in {W}^{1,p}(\mathbb R^n)$ belong to the family of  Gaussians, i.e., they are given by 
\begin{equation}\label{e-Gaussian}
		u_{\lambda,x_0}(x)=\lambda^\frac{n}{pp'}\left(\Gamma\left(\frac{n}{p'}+1\right)\omega_n\right)^{-\frac{1}{p}}e^{-\lambda\frac{|x+x_0|^{p'}}{p}},\ x\in \mathbb R^n, \ 
	\end{equation}
where $\lambda>0$ and $x_0\in \mathbb R^n$ are arbitrarily fixed.  When $p=2$, inequality \eqref{e-sharp-log-Sobolev} appears in the paper of Weissler \cite{Weissler}, which turns out to be equivalent to Gross' log-Sobolev inequality \cite{Gross} stated for the Gaussian measure. 
Very soon after the paper \cite{delPinoDolbeault-2} has been published, by using the Pr\'ekopa-Leindler inequality, Gentil \cite{Gentil} (see also \cite{Fujita}) showed  that in fact \eqref{e-sharp-log-Sobolev} is valid for all $p > 1$ and the equality in \eqref{e-sharp-log-Sobolev} holds true for functions of the type \eqref{e-Gaussian}. However, an 'if and only if' characterization of the equality in  \eqref{e-sharp-log-Sobolev}  remained still open for $p \geq n$. 

The purpose of the present paper is twofold. First, we prove a more general, {\it weighted}  version of \eqref{e-sharp-log-Sobolev} for \textit{all} values of $1 \leq p < \infty$ on open convex cones of $ \mathbb R^n$. Second,  we  provide an 'if and only if' characterization of the equality cases in our version of the $L^p$-log-Sobolev inequality. In particular, this result provides a concluding answer to the open problem concerning the characterization of the equality case for $p\geq n$ in $\mathbb R^n,$  see e.g. Gentil \cite[p.\ 599]{Gentil}.

To formulate our results we need to  introduce some notation. Let us consider $E\subseteq \mathbb R^n$, an open convex cone and $\omega\colon E\to (0,\infty)$ be a  log-concave homogeneous weight of class $\mathcal C^1$ with degree $\tau \geq 0$; that means that the function $x \to \log \omega(x) $ is concave on $E$ and it satisfies $\omega(\lambda x) = \lambda^{\tau} \omega(x)$ for all $x \in E$ and $\lambda \geq 0$. For $p \geq 1$  we denote by  $L^p(\omega;E)$ the space that contains all measurable functions $u: E \to \mathbb R$ such that $\displaystyle\int_{E} |u|^p \omega dx < \infty$ and we also introduce the weighted Sobolev space
\begin{equation}\label{W-p-q}
	{W}^{1,p}(\omega;E)=\{u\in L^p(\omega;E):\nabla u\in L^p(\omega;E)\}.	
\end{equation}
We denote by $B(x,r)\subset \mathbb R^n$ the ball with center in $x\in \mathbb R$ and radius $r>0$; in particular, we let $B\coloneqq B(0,1)$.  

As expected, the cases $p>1$ and $p=1$ should be treated separately.  We start with the main result for the case $p>1$: 




\begin{theorem}\label{log-Sobolev} $($Case $p>1$$)$
	Let $E\subseteq \mathbb R^n$ be an open convex cone and $\omega\colon E\to (0,\infty)$ be a  log-concave homogeneous weight of class $\mathcal C^1$ with degree $\tau\geq 0,$ and $p> 1.$ Then for every function $u\in {W}^{1,p}(\omega;E)$ with $\displaystyle\int_E |u|^p\omega dx=1$ we have 
	\begin{equation}\label{sharp-log-Sobolev}
	\mathcal E_{\omega,E}(|u|^p)\coloneqq 	\int_E |u|^p\log |u|^p \omega dx\leq \frac{n+\tau}{p}	\log\left(\mathcal L_{\omega,p}\displaystyle\int_E |\nabla u|^p\omega dx\right),
	\end{equation}
	where 
	$$\mathcal L_{\omega,p}=\frac{p}{n+\tau}\left(\frac{p-1}{e}\right)^{p-1}\left(\Gamma\left(\frac{n+\tau}{p'}+1\right)\int_{B\cap E}\omega\right)^{-\frac{p}{n+\tau}}.$$
	Equality holds in \eqref{sharp-log-Sobolev} if and only if the extremal function belongs to the family of Gaussians
	\begin{equation}\label{Gaussian}
		u_{\lambda,x_0}(x)=\lambda^\frac{n+\tau}{pp'}\left(\Gamma\left(\frac{n+\tau}{p'}+1\right)\int_{B\cap E}\omega\right)^{-\frac{1}{p}}e^{-\lambda\frac{|x+x_0|^{p'}}{p}},\ \  x\in E, \ \lambda>0,
	\end{equation}
where
\begin{itemize}
	\item $x_0\in -\partial E\cap \partial E$ and $\omega(x+x_0)=\omega(x)$ for every $x\in E$ whenever $\tau>0$, and 
	\item $x_0\in -\overline E\cap \overline E$ and $\omega$ is constant in $E$ whenever $\tau=0$. 
\end{itemize}
\end{theorem}
 Although the log-concavity gives a regularity of $\omega$ (i.e., an a.e. differentiability on $E$), we assume for simplicity that it is of class $\mathcal C^1$; a suitable modification of our proofs allows to relax this regularity assumption, see e.g.  Balogh,  Guti\'{e}rrez and  Krist\'{a}ly \cite{BGK_PLMS}. 

As we noticed, Theorem \ref{log-Sobolev}  provides an  extension of \eqref{e-sharp-log-Sobolev} to weighted Sobolev spaces as well as  a  full characterization of the extremal functions, answering in particular also the question from Gentil \cite{Gentil}. Furthermore,  Theorem \ref{log-Sobolev} gives the full class of extremal functions in the $L^p$-log-Sobolev inequality $(p>1)$ for monomial weights, see Nguyen \cite{Nguyen}, where only the 'if' part is stated.    

In the case $p=1$, the  framework for stating log-Sobolev inequality is the space of functions with weighted bounded variation $BV(\omega;E)$, rather than the Sobolev space ${W}^{1,1}(\omega;E)\subset BV(\omega;E)$. In fact, $BV(\omega;E)$ contains all functions $u \in L^1(\omega;E)$ such that 
$$ \sup \left\{ \int_E u \text{div}(\omega X)  dx: X \in C^1_c(E; \mathbb R^n) , |X| \leq 1 \right\} < \infty.$$ By the Riesz representation theorem we can associate to any function $u\in BV(\omega;E)$ its  weighted variation measure
$\|Du\|_{\omega}$ similarly as in the usual non-weighted case, see e.g. Evans and Gariepy \cite{EG}. Using this notation and $\mathbbm{1}_S$ for the characteristic function of the nonempty set $S\subset \mathbb R^n$, we can formulate the appropriate version of the weighted $L^1$-log-Sobolev inequality:

\begin{theorem}\label{log-Sobolev-1} $($Case $p=1$$)$
	Let $E\subseteq \mathbb R^n$ be an open convex cone and $\omega\colon E\to (0,\infty)$ be a  log-concave homogeneous weight of class $\mathcal C^1$ with degree $\tau\geq 0.$ Then for every function $u\in BV(\omega;E)$ with $\displaystyle\int_E |u|\omega dx=1$ we have 
	\begin{equation}\label{sharp-log-Sobolev-1}
		\mathcal E_{\omega,E}(|u|)\coloneqq 	\int_E |u|\log |u| \omega dx\leq (n+\tau)\log\left(\frac{\left(\displaystyle\int_{B\cap E}\omega\right)^{-{\frac{1}{n+\tau}}}}{n+\tau}\|D(|u|)\|_{\omega}(E)  \right).
	\end{equation}
Moreover, equality holds in \eqref{sharp-log-Sobolev-1} if and only if 
\begin{equation}\label{Indicator}
		u_{\lambda,x_0}(x)=\frac{\lambda^{-n-\tau}}{\displaystyle\int_{B \cap E} \omega dx} \mathbbm{1}_{B(-x_0, \lambda)\cap E}(x),\ \ \ x\in E, \ \lambda>0,
	\end{equation}
	where 
\begin{itemize}
	\item $x_0\in -\partial E\cap \partial E$ and $\omega(x+x_0)=\omega(x)$ for every $x\in B(-x_0, \lambda)\cap E$ whenever $\tau>0$, and 
	\item $x_0\in -\overline E\cap\overline E$ and $\omega$ is constant in $E$ whenever $\tau=0$. 
\end{itemize}
%
\end{theorem}

 Theorem \ref{log-Sobolev-1} extends the result of Beckner \cite{Beckner} and Ledoux \cite{Ledoux} to the weighted setting, providing also the full characterization of extremals. When $\omega\equiv 1$ and $E=\mathbb R^n$, Theorem \ref{log-Sobolev-1} follows by the quantitative log-Sobolev inequality established by   Figalli, Maggi and Pratelli  \cite{FMP2}. Furthermore, the results concerning monomial weights, treated by Nguyen \cite{Nguyen}, directly follows by Theorem \ref{log-Sobolev-1} together with the explicit form of the extremals.   

Log-Sobolev inequalities are typically  deduced by using a \textit{limiting argument} in a Gagliardo-Nirenberg-type inequality, see e.g.\ Beckner and Pearson \cite{Beckner-Pearson}, Del Pino and Dolbeault \cite{delPinoDolbeault-2}, Lam \cite{Lam}, Nguyen \cite{Nguyen} and references therein. In fact there are two different limiting arguments involved in this method: with respect to the \textit{parameter} in the Gagliardo-Nirenberg inequality and also with respect to \textit{approximation} by compactly supported smooth functions.  Because of the usage of  this double approximation,  an  'if and only if' characterization of the extremals is no longer possible by this method. 

Having as main goal the characterization of the extremals, we follow in this paper a different approach:  by using optimal transport theory  inspired by Cordero-Erausquin, Nazaret and Villani \cite{CENV}, \cite{Villani} -- where sharp Sobolev and Gagliardo-Nirenberg inequalities are established  --  we give a \textit{direct proof} of the log-Sobolev inequalities \eqref{sharp-log-Sobolev} and \eqref{sharp-log-Sobolev-1}.  In this way we avoid the use of the limiting argument with respect to the parameter in the Gagliardo-Nirenberg inequality. Next, we shall prove the inequality for a general  $W^{1,p}(\omega;E)$ resp.\ $BV(\omega;E)$ function to avoid the necessity of approximation with compactly supported smooth functions. However, an important technical difficulty arises: 
the proofs of inequalities \eqref{sharp-log-Sobolev} and \eqref{sharp-log-Sobolev-1}  are much simpler on $C_c^\infty(\mathbb R^n)$ because the usual tools of calculus are readily available for this class of functions. In order to handle this situation, we have to develop these tools of calculus for non-smooth functions. In this sense we shall prove two important technical propositions containing 'integration by part' formulas/inequalities for functions in the general  space $W^{1,p}(\omega;E)$ and $BV(\omega;E)$,  which will play crucial roles in the proof of  \eqref{sharp-log-Sobolev} and \eqref{sharp-log-Sobolev-1}, see Propositions \ref{prop:integrationbyparts} and \ref{prop:p=1}, respectively.   These formulas together with a scaling argument, a weighted entropy-type control (see Proposition \ref{proposition-key}) and  properties of optimal mass transport maps give not only an elegant proof for \eqref{sharp-log-Sobolev} and \eqref{sharp-log-Sobolev-1}, but also the full characterization of the equality cases by   
 tracking back the sharp estimates in the proofs.




The optimal mass transport theory has been successfully applied in proving weighted Sobolev and Gagliardo-Nirenberg inequalities, see e.g.\ Agueh,  Ghoussoub and  Kang \cite{AGK}, Ciraolo, Figalli and Roncoroni \cite{CFR}, Lam \cite{Lam}, Nguyen \cite{Nguyen},  and Balogh, Guti\'errez and Krist\'aly \cite{BGK_PLMS}. Moreover, Figalli, Maggi and Pratelli \cite{FMP, FMP2} applied this method to prove quantitative isoperimetric and Sobolev inequalities in the unweighted setting. Cordero-Erausquin \cite{Cordero}, Barthe and Kolesnikov \cite{BartheK} and Fathi, Indrei and Ledoux \cite{Fathi} gave a direct proof of the Gaussian log-Sobolev inequality (even in quantitative form) for $p=2$  based on optimal mass transport theory. However, we are not aware of the existence of a direct proof of the sharp Euclidean $L^p$-log-Sobolev inequality based on the optimal transport method, $p>1$.

We notice that -- by using the ABP-method --  further sharp Sobolev inequalities with weights have been also stated by Cabr\'e, Ros-Oton and Serra \cite{Cabre-Ros-Oton, Cabre-Ros-Oton-Serra}. Similarly to the aforementioned works (see  \cite{CFR},  \cite{Lam},  \cite{Nguyen} and \cite{BGK_PLMS}), in the papers \cite{Cabre-Ros-Oton, Cabre-Ros-Oton-Serra} the initial assumption is the concavity of $\omega^\frac{1}{\tau}$ on $E\subseteq \mathbb R^n$ (with the convention that $\omega$ is constant whenever $\tau=0$). Even though the concavity of $\omega^\frac{1}{\tau}$ implies generically the log-concavity of $\omega$, -- the latter being crucial in the proof of the log-Sobolev inequalities -- 
one can observe that the converse also holds due to the homogeneity of $\omega$,   see Proposition \ref{prop-log-concave}. Consequently, the assumptions of Theorems \ref{log-Sobolev} and \ref{log-Sobolev-1} are in a full concordance with the ones used in the usual literature, and they essentially require the weight  $\omega$ to verify  the curvature-dimension condition ${\sf CD}(0,n+\tau)$ on $E$, see \cite{Cabre-Ros-Oton-Serra}. 

In several of these  works the authors established Sobolev-type inequalities with general norms instead of the Euclidean one, see e.g.\ Cordero-Erausquin, Nazaret and Villani \cite{CENV},   Cabr\'e, Ros-Oton and Serra  \cite{Cabre-Ros-Oton-Serra}, Figalli, Maggi and Pratelli \cite{FMP}, Gentil \cite{Gentil}, Lam \cite{Lam}, Nguyen \cite{Nguyen}. Although Theorems \ref{log-Sobolev} and \ref{log-Sobolev-1} can be stated in the same general context, we avoid to do this for the sake of simplicity. Such an anisotropic extension relies on the definition of the dual norm $\|\cdot\|_*$ of an arbitrary norm $\|\cdot\|$ given by $\|X\|_*=\sup_{\|Y\|\leq 1}X\cdot Y$, where $X\cdot Y$ stands for the usual Euclidean inner product of $X,Y\in \mathbb R^n$, see Remark \ref{remark-anistropic}. It is worth to notice that this  extension does not require any modification of the optimal transport theory of Brenier-McCann (see  \cite{Brenier} and \cite{McCann}), the cost function still remaining the Euclidean one. 

As an application of the sharp $L^p$-log-Sobolev inequality (Theorem \ref{log-Sobolev}), we  establish a \textit{sharp weighted hypercontractivity estimate} for the Hopf-Lax semigroup related to Hamilton-Jacobi equation and characterize the equality cases. To be more precise, for $p>1$ and for a function  $g: E\to \mathbb R$ we recall the Hopf-Lax formula 
\begin{equation}\label{inf-convolution-0}
	{\bf Q}_{t}g(x)\coloneqq {\bf Q}_{t}^pg(x)=\inf_{y\in E}\left\{g(y)+\frac{|y-x|^{p'}}{p't^{p'-1}}\right\}, \ t\in (0,t_0),x\in E.
\end{equation}
Here $t_{0}>0$ is chosen such that ${\bf Q}_{t}g(x) \in \mathbb R$  for  $t \in (0, t_{0})$ and $x\in E$ (by convention, we consider ${\bf Q}_0g=g$). Clearly, $t_{0}= \infty$  whenever $g$ is a bounded Lipschitz function on $E$. Moreover, $(x,t)\mapsto {\bf Q}_{t}g(x)$ solves the classical Hamilton-Jacobi equation; see e.g.\ Gentil \cite{Gentil} or the book of Evans \cite{Evans}.  However, in order to characterize the equality cases in the hypercontractivity estimate (see Theorem \ref{theorem-main-3} below), we need a larger class of functions $g$ than bounded and Lipschitz,  that eventually contains the candidates for extremals in the hypercontractivity estimate. To achieve this, for a fixed $t_0>0$, we consider the family of functions
$$\mathcal  F_{t_0}(E)\coloneqq  \left\{ g:E\to \mathbb R:\begin{array}{lll}
	g {\rm \ is\ measurable,\ bounded\ from\ above\ and}\  \\
	{\rm there\ exists}\ x_0\in E {\rm \ such \ that} \ {\bf Q}_{t_0}g(x_0)>-\infty
\end{array}\right\}.
$$
In the sequel we show that if $g \in \mathcal  F_{t_0}(E)$ then ${\bf Q}_{t}g(x) \in \mathbb R$ for every $t\in (0,t_0)$ and $x\in E$, and the following weighted hypercontractivity estimate holds.

\begin{theorem}\label{theorem-main-3}
	Let $E\subseteq \mathbb R^n$ be an open convex cone and $\omega\colon E\to (0,\infty)$ be a  log-concave homogeneous weight of class $\mathcal C^1$ with degree $\tau\geq 0,$ and $p> 1.$ Given the numbers $t_0>0$ and $0<\alpha\leq \beta$, for every function $g\in \mathcal F_{t_0}(E)$ with $e^g\in  L^\alpha(\omega;E)$ and every $\tilde t\in (0,t_0)$ we have
	\begin{equation}\label{hyperc-estimate}
		\|e^{{\bf Q}_{\tilde t}g}\|_{L^{\beta}(\omega;E)}\leq \|e^{g}\|_{L^{\alpha}(\omega;E)}\left(\frac{\beta-\alpha}{\tilde t}\right)^{\frac{n+\tau}{p}\frac{\beta-\alpha}{\alpha\beta}}\frac{\alpha^{\frac{n+\tau}{\alpha\beta}(\frac{\alpha}{p}+\frac{\beta}{p'})}}{\beta^{\frac{n+\tau}{\alpha\beta}(\frac{\beta}{p}+\frac{\alpha}{p'})}}\left((p')^\frac{n+\tau}{p'}\Gamma\left(\frac{n+\tau}{p'}+1\right)\int_{B \cap E}\omega\right)^{\frac{\alpha-\beta}{\alpha\beta}}.
	\end{equation} 
	In addition, equality holds in \eqref{hyperc-estimate} for some $\tilde t\in (0,t_0)$ and $g\in \mathcal F_{\tilde t}(E)$ with $e^g\in  L^\alpha(\omega;E)$  if and only if $\alpha<\beta$ and 
	\begin{equation}\label{extremal-hyperc}
		g(x)=C-\frac{1}{p'}\left(\frac{\beta-\alpha}{\beta\tilde t}\right)^\frac{1}{p-1}|x+x_0|^{p'},\ \ x\in E,
	\end{equation}
for some $C\in \mathbb R$ and 
	\begin{itemize}
		\item $x_0\in -\partial E\cap \partial E$ and $\omega(x+x_0)=\omega(x)$ for every $x\in  E$ whenever $\tau>0$, and 
		\item $x_0\in -\overline E\cap\overline E$ and $\omega$ is constant in $E$ whenever $\tau=0$. 
	\end{itemize}
\end{theorem}

 The proof of Theorem \ref{theorem-main-3} is based on the a.e.\ validity of the Hamilton-Jacobi equation associated with the Hopf-Lax solution $(x,t)\mapsto {\bf Q}_{t}g(x)$  and its integrability properties (see Proposition \ref{proposition-hyper}) combined with the weighted $L^p$-log-Sobolev inequality. The  equality cases in \eqref{hyperc-estimate} are obtained via the characterization of the extremal functions in the  $L^p$-log-Sobolev inequality formulated in Theorem \ref{log-Sobolev} and by using properties of the map  $(x,t)\mapsto {\bf Q}_{t}g(x)$.

The paper is organized as follows. In Section \ref{section-2} we present some basic properties of the weights, including a weighted entropy-type integrability, based on a two-weighted Sobolev inequality established in \cite{BGK_PLMS}. 
In Section \ref{section-3} we prove Theorem \ref{log-Sobolev}, where one of the crucial steps is an integration by parts formula (see Propositions \ref{prop:integrationbyparts}) for functions belonging to $W^{1,p}(\omega;E)$, $p>1$. Section \ref{section-4} is a kind of counterpart of  Section \ref{section-3}, where we prove Theorem \ref{log-Sobolev-1}, and one of the main ingredients is again an integration by parts formula (see Proposition \ref{prop:p=1}) for functions in  $BV(\omega;E)$.
In Section \ref{section-5} we present the proof of Theorem \ref{theorem-main-3}.

\section{Preparatory results} \label{section-2}
We start with an elementary observation related to (log-)concave homogeneous weights. 
 
\begin{proposition}\label{prop-log-concave}
	Let $E\subseteq \mathbb R^n$ be an open convex cone and $\omega\colon E\to (0,\infty)$ be a homogeneous weight of class $\mathcal C^1$ with degree $\tau\geq 0.$ 
	
	\begin{itemize}
		\item[(i)] 
		If $\tau> 0,$ the following statements are equivalent: 
		\begin{itemize}
			\item[(a)]  $\omega$ is log-concave in $E;$
			\item[(b)]   for every $x,y\in E$ one has 
			\begin{equation}\label{log-concave-charact}
				\log\left(\frac{\omega(y)}{\omega(x)}\right)\leq -\tau+\frac{\nabla \omega (x)\cdot y}{\omega (x)};
		\end{equation}
		\item[(c)] $\omega^\frac{1}{\tau}$ is concave in $E.$

		\end{itemize}
	Moreover, any of the above properties imply that  $0\leq \nabla \omega(x)\cdot y$ for every $x,y\in E$.  
	\item[(ii)] 
	If  $\omega$ is log-concave in $E$, then $\tau=0$  if and only if $\omega$ is constant in $E;$ in particular, if $E=\mathbb R^n$ then $\omega$ is constant in $\mathbb R^n$.
	\end{itemize}

\end{proposition}


\begin{proof} 
	(i) We assume that $\tau>0$. The equivalence between (a) and (b) is a consequence of the definition and the Euler's relation: $\nabla \omega(x) \cdot x = \tau \omega(x)$ for every $x\in E$ for the $\tau$-homogeneous function $\omega$. 
	
	The implication (c)$\implies$(a) is generically valid even in the absence of homogeneity. 
	
	 We now prove the implication  (a)$\implies$(c). For simplicity, let $v\coloneqq \omega^\frac{1}{\tau}$, which is log-concave in $E$ by assumption. 
	 Let $x_1,x_2\in E$ be two fixed elements and $v(x_1)=a_1>0$ and $v(x_2)=a_2>0.$ By the 1-homogeneity of $v$, one has that $v(x_1/a_1)=v(x_2/a_2)=1$, thus the log-concavity of $v$ implies that for every $t\in [0,1]$ we have
	$$v\left(t\frac{x_1}{a_1}+(1-t)\frac{x_2}{a_2}\right)\geq v^t\left(\frac{x_1}{a_1}\right)v^{1-t}\left(\frac{x_2}{a_2}\right)=1.$$
	If $\lambda\in [0,1]$ is arbitrarily fixed, let $t\coloneqq \frac{\lambda a_1}{\lambda a_1+(1-\lambda )a_2}\in [0,1]$ be in the above estimate; then, we obtain 
	$$v\left(\frac{\lambda x_1+(1-\lambda)x_2}{\lambda a_1+(1-\lambda )a_2}\right)\geq 1.$$
	  Again by the $1$-homogeneity of $v$, the latter relation becomes equivalent to $$v(\lambda x_1+(1-\lambda)x_2)\geq \lambda a_1+(1-\lambda )a_2=\lambda v(x_1)+(1-\lambda )v(x_2),$$ which is precisely the concavity of $v=\omega^\frac{1}{\tau}$. 
	
	By contradiction, let us assume that  there exist $x_0,y_0\in E$ such that $ \nabla \omega(x_0)\cdot y_0<0.$ We choose $x\coloneqq x_0$ and $y\coloneqq \lambda y_0$ with $\lambda>0$ in \eqref{log-concave-charact}; by the homogeneity of $\omega$, it turns out that 
	$$\tau \log \lambda + \log\left(\frac{\omega(y_0)}{\omega(x_0)}\right)\leq -\tau+\lambda\frac{\nabla \omega (x_0)\cdot y_0}{\omega (x_0)},\ \ \lambda>0.$$
	Letting $\lambda \to \infty,$ we get a contradiction. 
	
(ii)	If $\omega$ is constant, then trivially $\tau=0$. Conversely, if $\tau=0$, we have $\omega(\lambda y)=\omega(y)$ for every $\lambda>0$ and $y\in E$. Clearly, \eqref{log-concave-charact} is also valid (with $\tau=0$); in particular, by putting $y\coloneqq \lambda y$,  we obtain 
	$$	\log\left(\frac{\omega(y)}{\omega(x)}\right)\leq \lambda\frac{\nabla \omega (x)\cdot y}{\omega (x)}\ \ {\rm for\ all}\ \lambda>0,\ x,y\in E.$$
	Letting $\lambda\to 0$, the latter inequality implies that $\omega(y)\leq \omega(x)$ for every $x,y\in E$, which shows that $\omega$ is constant in $E$.

	Let $E=\mathbb R^n$. Since $0\leq \nabla \omega(x)\cdot y$ for every $x,y\in \mathbb R^n$ (similarly to the case (i)),  we necessarily have $\nabla \omega(x)=0$ for every $x\in \mathbb R^n$, i.e., $\omega$ is constant in $\mathbb R^n$. 
\end{proof}

\begin{remark}\rm 
	Large classes of weights verifying the properties of Proposition \ref{prop-log-concave} can be found in Cabr\'e, Ros-Oton and Serra \cite[Section 2]{Cabre-Ros-Oton-Serra} and Balogh, Guti\'errez and Krist\'aly \cite[Section 4]{BGK_PLMS}, including -- among others -- the monomial weight $\omega\colon E\to (0,\infty)$  defined as 
	$\omega(x)=x_1^{a_1}\dots x_n^{a_n}$ with $a_i\geq 0$ for every $i=1,...,n$ and $$E=\left\{x=(x_1,...,x_n)\in \mathbb R^n: x_i>0\ {\rm whenever}\ a_i>0\right\}.$$
	
\end{remark}

The following auxiliary result concerns a weighted entropy-type estimate which is crucial in the proof of our main theorems. To prove it, we need a two-weighted Sobolev inequality by Balogh, Guti\'errez and Krist\'aly \cite{BGK_PLMS} that we recall in the sequel.   

Let $E\subseteq \mathbb R^n$ be an open convex cone.
First, we assume that the two weights $\omega_1,\omega_2:E\to (0,\infty)$ satisfy the homogeneity condition \begin{equation}\label{eq:homogeneity of weights}
	\omega_i(\lambda\, x)=\lambda^{\tau_i}\,\omega_i(x)\ \ \textrm{for all}\ \lambda>0, x\in E,
\end{equation}
where the  parameters $\tau_1,\tau_2\in\mathbb R$ verify 
\begin{equation}\label{eq:p-range}
	1\leq  s<\tau_2+n\leq\tau_1+s+n,
\end{equation}
and
\begin{equation}\label{eq:p-range-bis}
	\tau_2 \geq \left(1-\frac{s}{n}\right)\tau_1,
\end{equation}
for some $s\in \mathbb R.$
Furthermore, for the given values $s, \tau_1, \tau_2$ as above choose  $t\in \mathbb R$  such that the following balance condition holds:
\begin{equation}\label{eq:dimensional balance}
	\dfrac{\tau_1+n}{t}=\dfrac{\tau_2+n}{s}-1.
\end{equation}
By relations \eqref{eq:dimensional balance} and (\ref{eq:p-range}), we obtain that 
$t=\frac{s(\tau_1+n)}{\tau_2+n-s}\geq s.$
The fractional dimension $n_a$ is given by  
\begin{equation} \label{eq:fracdim}
	\dfrac{1}{n_a}=\dfrac{1}{s}-\dfrac{1}{t}.
\end{equation}
By \eqref{eq:dimensional balance}, relation (\ref{eq:p-range-bis}) turns out to be equivalent to 
$n_a\geq n.$ The cases $n_a= n$ and $n_a>n$ should be discussed separately, both of them being crucial in our further investigations. It turns out that the weighted Sobolev inequality 
\begin{equation}
	\left(\int_E |v(x)|^t\,\omega_1(x) \,dx\right)^{1/t}\leq  K_0\,\left(\int_E |\nabla v(x)|^s\,\omega_2(x)\,dx\right)^{1/s}\ \ \textrm{for all}\ v\in C_c^\infty(\mathbb R^n), \tag{WSI}\label{WSI}
\end{equation}
holds, see  \cite[Theorem 1.1]{BGK_PLMS}, where $K_0>0$ is independent on $v\in C_c^\infty(\mathbb R^n)$, whenever
\begin{itemize}
	\item[(I)] either $n_a= n$ together with
	$\sup_{x\in E}\dfrac{\omega_1(x)^{1/t}}{\omega_2(x)^{1/s}}=:C\in (0,\infty)$ and
	\begin{equation}\label{c-1-feltetel}
		0
		\leq
		\left(\dfrac{1}{s'}\dfrac{\nabla \omega_1(x)}{\omega_1(x)}+\dfrac{1}{s}\dfrac{\nabla \omega_2(x)}{\omega_2(x)}\right)\cdot y
	\end{equation}
	for a.e. $x\in E$ and for all $y\in E$,
	\item[(II)] or $n_a> n$ and  there exists a constant $C_0>0$ such that 
	\begin{equation} \label{c-0-condition}
		\left(\left(\frac{\omega_2(y)}{\omega_2 (x)}\right)^\frac{1}{s}\left(\frac{\omega_1(x)}{\omega_1 (y)}\right)^\frac{1}{t}\right)^\frac{n_a}{n_a-n}\leq C_0 \left(\dfrac{1}{s'}\dfrac{\nabla \omega_1(x)}{\omega_1(x)}+\dfrac{1}{s}\dfrac{\nabla \omega_2(x)}{\omega_2(x)}\right)\cdot y,
	\end{equation} 
	for a.e. $x\in E$ and for all $y\in E$. 
\end{itemize}
Here, as before,  $s'=s/(s-1)$ denotes the conjugate of $s\geq 1$ (with $s'=+\infty$ if $s=1$).  Moreover, by density arguments, the inequality \eqref{WSI} is valid for functions $v:\mathbb R^n\to \mathbb R$ belonging to ${W}^{1,s}(\omega_2;E)$.  Note that by scaling arguments, \eqref{WSI} necessarily requires the dimensional balance condition  \eqref{eq:dimensional balance}.


After this preparatory part, we state the following result. 

\begin{proposition}\label{proposition-key}
	Let $p\geq 1$, $E\subseteq \mathbb R^n$ be an open convex cone and $\omega\colon E\to (0,\infty)$ be a  log-concave homogeneous weight of class $\mathcal C^1$ with degree $\tau\geq 0.$  Then for every $u\in {W}^{1,p}(\omega;E)$ such that $\displaystyle\int_E |u|^p\omega dx=1$  the following inequalities hold$:$
	\vspace{-0.3cm}
	\begin{equation} \label{ineq_1}
	\int_E |u|^p\omega \log(|u|^p\omega)<+\infty,
	\end{equation}
	\begin{equation} \label{ineq_2} 
	\int_E |u|^p\omega \log|u|^p<+\infty,
	\end{equation}
	\begin{equation} \label{ineq_3} 
	-\infty < \int_E |u|^p\omega \log \omega.
	\end{equation}
	In addition, for $p=1$, the same statements hold for functions  $u\in BV(\omega;E)$. 
\end{proposition}
{\it Proof.} We prove relation \eqref{ineq_1} first. We consider only the case when $\tau>0$; otherwise, if $\tau=0$, the function $\omega$ is constant and the claim holds by the Euclidean log-Sobolev inequality \eqref{e-sharp-log-Sobolev}. 
Since  $n\geq 2$, we may choose  the parameters $s=1$, $t= \frac{n}{n-1}$
and the weights $\omega_1=\omega^{\frac{n}{n-1}}$,  $ \omega_2=\omega.$
In particular, we have $\tau_1=\frac{n}{n-1}\tau$ and $\tau_2=\tau.$ Moreover,  simple computations confirm that the assumptions \eqref{eq:homogeneity of weights}--\eqref{eq:dimensional balance} are verified. 
Note that for the fractional dimension we have
$	\dfrac{1}{n_a}=\dfrac{1}{s}-\dfrac{1}{t}=\dfrac{1}{n},$
thus $n_a=n$. In addition, 
$$\dfrac{\omega_1(x)^{1/t}}{\omega_2(x)^{1/s}}=\dfrac{\omega(x)}{\omega(x)}=1,\ \ x\in E.$$
Finally, it remains to verify \eqref{c-1-feltetel}, which directly follows by Proposition \ref{prop-log-concave}. 
According to the above reasons, -- keeping the above notations, -- it follows by (I) that the  weighted Sobolev inequality \eqref{WSI}  holds, i.e., 
\begin{equation}\label{Sobolev-modified}
		\left(\int_E |v(x)|^{\frac{n}{n-1}}\,\omega(x)^{\frac{n}{n-1}} \,dx\right)^{\frac{n-1}{n}}\leq  K_0\,\int_E |\nabla v(x)|\,\omega(x)\,dx\ \ \textrm{for all}\ v\in {W}^{1,1}(\omega;E).
\end{equation}

 Fix $u\in {W}^{1,p}(\omega;E)$ such that 
$\displaystyle\int_E |u|^p\omega dx=1$ and let $v\coloneqq |u|^p.$ We  prove that $v\in {W}^{1,1}(\omega;E).$ 
Let us assume first that $p>1$.  We observe that $$	\int_E | v(x)|\,\omega(x)\,dx=\int_E | u(x)|^p\,\omega(x)\,dx=1.$$
Moreover, since $\nabla v=p|u|^{p-1}\nabla |u|$, by H\"older's inequality and  $ u\in {W}^{1,p}(\omega;E)$ we have that  
\begin{eqnarray*}
	\int_E |\nabla v(x)|\,\omega(x)\,dx&=&p\int_E|\nabla u(x)| | u(x)|^{p-1}\,\omega(x)\,dx\\&\leq& p\left(\int_E |\nabla u(x)|^{p}\,\omega(x)\,dx\right)^\frac{1}{p}\left(\int_E | u(x)|^{p}\,\omega(x)\,dx\right)^{\frac{1}{p'}}<+\infty.
\end{eqnarray*}
Therefore, $ v\in {W}^{1,1}(\omega;E)$. Now, by \eqref{Sobolev-modified} we obtain 
$$\int_E |u(x)|^{p\frac{n}{n-1}}\,\omega(x)^{\frac{n}{n-1}} \,dx=\int_E |v(x)|^{\frac{n}{n-1}}\,\omega(x)^{\frac{n}{n-1}} \,dx <+\infty.$$
By Jensen's inequality and the above estimate we have that
$$\frac{1}{n-1} \displaystyle \int_E |u|^p\omega \log(|u|^p\omega)=\displaystyle \int_E |u|^p\omega \log\left((|u|^p\omega)^\frac{1}{n-1}\right)\leq \log\left(\int_E |u(x)|^{p\frac{n}{n-1}}\,\omega(x)^{\frac{n}{n-1}}\right)<+\infty,$$
which concludes the proof of  \eqref{ineq_1}.

The proof of \eqref{ineq_2} conceptually is similar to the above argument. In this case we choose the parameters $s=1$, $t= \frac{\tau+n}{\tau+n-1}$ and the weights $\omega_1 = \omega_2 = \omega$; therefore, 
$\tau_1=\tau_2=\tau$ and the fractional dimension is $n_a=n+\tau>n.$ Moreover, due to Proposition \ref{prop-log-concave}, inequality \eqref{c-0-condition} becomes equivalent to the concavity of $\omega^\frac{1}{\tau}$  once we choose  $C_0=\frac{1}{\tau}$. Therefore, by (II)  the Sobolev inequality  \eqref{WSI} reads as
\begin{equation} \label{Sobolev-modified-2}
		\left(\int_E |v(x)|^ \frac{\tau+n}{\tau+n-1}\,\omega(x) \,dx\right)^{ \frac{\tau+n-1}{\tau+n}}\leq  K_0\,\int_E |\nabla v(x)|\,\omega(x)\,dx\ \ \textrm{for all}\ v\in {W}^{1,1}(\omega;E).
\end{equation}
In the same way as before we can combine Jensen's inequality and the estimate \eqref{Sobolev-modified-2} to deduce that
$$\frac{1}{\tau+n-1} \displaystyle \int_E |u|^p\omega \log|u|^p=\displaystyle \int_E |u|^p\omega \log\left((|u|^p)^\frac{1}{\tau+n-1}\right)\leq \log\left(\int_E |u(x)|^{p\frac{\tau+n}{\tau+n-1}}\,\omega(x)\right)<+\infty,$$
which concludes the proof of \eqref{ineq_2}.

When $p=1$, by standard approximation arguments  the inequality \eqref{Sobolev-modified} extends to functions $v\in  BV(\omega;E)$ by replacing the right hand side to the $\omega$-bounded variation of $v\in  BV(\omega;E)$, see \cite[Theorem 5.1]{BBF}.  The rest of the proof of inequalities \eqref{ineq_1} and \eqref{ineq_2}  is similar to the case $p>1$. 

Finally we prove inequality \eqref{ineq_3}. Let us fix a point $x_0 \in E$. By the log-concavity of $\omega$ (see \eqref{log-concave-charact}) we can write 
$$ \log \omega(x_0) +\tau \leq \log \omega(x) + \frac{\nabla \omega(x)}{\omega(x)}\cdot x_0  \ \ \text{for all} \ x \in E.$$
Multiplying by $u^p \omega$ the latter inequality and integrating over $E$ we obtain 
$$(\log \omega(x_0) + \tau) \int_E |u|^p \omega  - \int_E |u|^p \nabla\omega(x) \cdot x_0 \leq \int_E |u|^p \omega \log \omega.$$
Since $\displaystyle\int_E |u|^p\omega dx=1$,   we only need to check that 
$$ \int_E |u|^p \nabla\omega(x) \cdot x_0 < \infty.$$
To verify this, we integrate by parts to obtain
$$ \int_E |u|^p \nabla\omega(x) \cdot x_0 \, dx= \int_{\partial E} |u|^p \omega x_0\cdot \textbf{n} \, d\mathcal H^{n-1} -p\int_E |u|^{p-1}(\nabla|u| \cdot x_0) \omega \,dx ,$$
where $\textbf{n}(x)$ is the unit outer normal vector to $E$ at $x\in \partial E$. 
Since $x_0 \in E$, by the convexity of the cone $E$ we have that the vectors $x_0$ and $\textbf{n}(x)$ 
 form and obtuse angle, i.e., 
$x_0 \cdot \textbf{n}(x) \leq 0$ for all points $x\in \partial E$ where the normal vector $\textbf{n}(x)$ exists;  thus the first integral on the right hand side is non-positive. Therefore, it is enough to estimate the second integral from above, which is obtained  by H\"older's inequality, i.e.,   
$$-\int_E |u|^{p-1}(\nabla|u| \cdot x_0) \omega \,dx \leq  |x_0| \int_E |u|^{p-1} |\nabla u| \omega \,dx \leq 
|x_0| \left(\int_E |u|^p \omega\right)^{\frac{1}{p'}} \left(\int_E|\nabla u|^p\omega\right)^{\frac{1}{p}} < \infty,$$
finishing the proof in the case $p>1$. If $p=1$, we do not need to apply H\"older's inequality but we need to perform an approximation argument and work with the $\omega$-bounded variation of measure as indicated before. 
\hfill $\square$\\

	
\section{Proof of Theorem \ref{log-Sobolev}: the case $p>1$} \label{section-3}

The main idea of the application of the optimal mass transport method in proving geometric inequalites lies in the fact that the optimal transport map is given as the gradient of a convex function $\phi: E \to \mathbb R$. We recall that by the theorem of Alexandrov \cite{EG} a convex function is twice differentiable almost everywhere. We denote by 
$\Delta_{\mathcal D'}\phi$ the distributional Laplacian of $\phi$; this is a positive Radon measure whose absolutely continuous part is denoted by $\Delta_A\phi$ and is called the Alexandrov Laplacian of $\phi$. 
Clearly  the inequality $$\Delta_A\phi\leq \Delta_{\mathcal D'}\phi $$
holds in the sense of distributions. 

Before presenting the proof of Theorem \ref{log-Sobolev}, we need an integration by part formula for functions $u\in W^{1,p}(\omega;E)$ whose proof is inspired by  Cordero-Erausquin, Nazaret and Villani \cite{CENV} and Nguyen \cite{Nguyen}. 

\begin{proposition}\label{prop:integrationbyparts}
	Let $p>1$, $u\in W^{1,p}(\omega;E)$ be non-negative,  $\phi\colon \mathbb R^n\to \mathbb R$ be a convex function such that $u^{p-1}\nabla\phi\in L^{p'}(\omega;E)$ and  $\nabla\phi(x)\in \overline E$ for a.e. $x\in \overline E$. 	
	Then we have
	\begin{equation}\label{eq:Byparts}
		\int_E u^p\omega \Delta_A\phi dx\leq -p\int_E u^{p-1}\omega \nabla \phi \cdot \nabla u dx -\int_E u^p \nabla \omega \cdot \nabla \phi dx.
	\end{equation}

\end{proposition}
\begin{proof}
	Let $S\coloneqq  {\rm int}\{x\in \mathbb R^n: \phi(x)<+\infty\}$; by assumption, $E\cap S$ is open, convex and contains the support of $u$ except, possibly, for a Lebesgue null set.
	
	For every $k\geq 1$, let $\theta_k^0:[0,\infty)\to [0,1]$ be a non-decreasing $\mathcal C^\infty$ function such that $\theta_k^0(s)= 0$ for $s\leq \frac{1}{2k}$ and $\theta_k^0(s)= 1$ for $s\geq \frac{1}{k}$; in addition, let $\theta_k:E\to [0,1]$  be the locally Lipschitz function $\theta_k(x)=\theta_k^0(d(x,\partial E))$, $x\in E.$ 
	Let $u_k=u\theta_k$, $k\geq 1$. A simple argument shows that $u_k\in W^{1,p}(\omega;E)$ for every $k\geq 1$ and ${\rm supp}u_k\subseteq S_k$ where $S_k=\{x\in E:d(x,\partial E)\geq \frac{1}{2k}\}.$  Furthermore, for every fixed $k\geq 1$, by a trivial extension, the function  $u_k$  belongs to the usual Sobolev space $W^{1,p}(\mathbb R^n)$ as well.

	%

	%
	

	Consider now a cut-off function $\chi\colon \mathbb R^n\to [0,1]$ of class $\mathcal C^\infty$ such that $\chi\equiv 1$ on $B=B(0,1)$, $\chi\equiv 0$ on $\mathbb R^n\setminus B(0,2)$. 
	Fix $x_0\in E$ and $k\geq 1$. For 	$\varepsilon\in(0,1)$  small enough, we define 
	\[
	u_{k,\varepsilon}(x)=\min\left\{u_k\left(x_0+\frac{x-x_0}{1-\varepsilon}\right),\chi(\varepsilon x)u_k(x)\right\},\  \ x\in E.
	\]
	Then $u_{k,\varepsilon}\geq 0$ has compact support in $E\cap S$ and since $|\nabla\chi|\leq C_0$ for some $C_0>0$, one has that $u_{k,\varepsilon} \in W^{1,p}(\omega;E)\cap W^{1,p}(\mathbb R^n)$. 
	
	Given $\kappa\in C_c^\infty(B)$ be a non-negative function such that $\displaystyle\int_{B} \kappa  dx=1$, we consider the convolution kernel $\kappa_\delta(x)=\frac{1}{\delta^{n}}\kappa( x/\delta)$, $\delta>0,$  and define the convolution function  $$u_{k,\varepsilon}^\delta(x)=(u_{k,\varepsilon}\star \kappa_\delta)(x)\coloneqq \int_E u_{k,\varepsilon}(y)\kappa_\delta(x-y) dy\geq 0.$$
	It is standard that the functions $u_{k,\varepsilon}^\delta$ are smooth on $E$ for every  $\varepsilon\in (0,1)$ belonging to the above range, and $\delta>0$, and $u_{k,\varepsilon}^\delta$ converges to $u_{k,\varepsilon}$ in $W^{1,p}(\omega;E)$ as $\delta\to 0$.  Moreover, since $\phi$ is convex, and ${\rm supp} u_{k,\varepsilon}^\delta$ is compact in $E$, then $\nabla\phi$ is essentially bounded on ${\rm supp}(u_{k,\varepsilon}^\delta)$ for every sufficiently small $\delta>0$.
	
	The idea is to apply for the smooth function $u_{k,\varepsilon}^\delta$ the usual divergence theorem and pass to the limit as $\delta\to 0$, $\varepsilon \to 0$ and $k\to \infty$, in this order, obtaining the expected inequality \eqref{eq:Byparts}.  
	
	Taking into account that $u_{k,\varepsilon}^\delta=0$ on $\partial E$, and $\Delta_A\phi\leq \Delta_{\mathcal D'}\phi$ in the distributional sense, the divergence theorem implies 
		\begin{eqnarray}\label{eq:GaussGreen}
			\nonumber	\int_E (u_{k,\varepsilon}^\delta)^p\omega \Delta_A\phi dx&\leq& \int_E (u_{k,\varepsilon}^\delta)^p\omega \Delta_{D'}\phi dx\\&=& -p\int_E (u_{k,\varepsilon}^\delta)^{p-1}\nabla u_{k,\varepsilon}^\delta\cdot \nabla \phi \omega dx-\int_E (u_{k,\varepsilon}^\delta)^p\nabla \omega \cdot \nabla \phi dx.
		\end{eqnarray}
	Thus, by inequality \eqref{eq:GaussGreen} we have
	\begin{equation}\label{eq:GaussGreen2}
		\int_E (u_{k,\varepsilon}^\delta)^p\left(\omega \Delta_A\phi +\nabla \omega \cdot \nabla \phi\right)dx\leq-p\int_E (u_{k,\varepsilon}^\delta)^{p-1}\nabla u_{k,\varepsilon}^\delta\cdot \nabla \phi \omega dx.
	\end{equation}
	
Since $u_{k,\varepsilon}^\delta$ converges to $u_{k,\varepsilon}$ in $W^{1,p}(\omega;E)$, one has that 
	$\nabla u_{k,\varepsilon}^\delta$ converges to $\nabla u_{k,\varepsilon}$ in $L^p(\omega;E)$  and $(u_{k,\varepsilon}^\delta)^{p-1}$ converges to $(u_{k,\varepsilon})^{p-1}$ in $L^{p'}(\omega;E)$ whenever $\delta\to 0$, the latter property coming from the property of superposition operators, see e.g.\ Willem \cite[Appendix A]{Willem}. In addition, since 
	$\nabla \phi$ is essentially bounded on ${\rm supp}(u_{k,\varepsilon}^\delta)$, we have that
	\[
	\int_E (u_{k,\varepsilon}^\delta)^{p-1}\nabla u_{k,\varepsilon}^\delta\cdot \nabla \phi \omega dx\rightarrow \int_E (u_{k,\varepsilon})^{p-1}\nabla u_{k,\varepsilon}\cdot \nabla \phi \omega dx \quad \text{as $\delta\to0$}.
	\]

%

Since $\phi$ is convex, then $\Delta_A\phi\geq 0$; moreover, by Proposition \ref{prop-log-concave} we also have that   $\nabla\phi\cdot \nabla\omega\geq 0$ a.e.\ on $E$, since $\nabla\phi(x)\in \overline E$ for a.e.\ $x\in \overline E$. Therefore, as $\delta\to 0$, the latter properties together with \eqref{eq:GaussGreen2} and  Fatou's lemma imply  that
	\begin{equation}\label{eq:GaussGreen3}
		\int_E u_{k,\varepsilon}^p\left(\omega \Delta_A\phi +\nabla \omega \cdot \nabla \phi\right)dx\leq-p\int_E u_{k,\varepsilon}^{p-1}\nabla u_{k,\varepsilon}\cdot \nabla \phi \omega dx.
	\end{equation}
	
	In the sequel we pass to the limit in \eqref{eq:GaussGreen3} as $\varepsilon\to 0.$ By the definition of $u_{k,\varepsilon}$, we first observe that --  extracting eventually a subsequence of $\varepsilon=(\varepsilon_{k,l})_{l\in \mathbb N}$ -- the sequence $u_{k,\varepsilon}$ converges to $u_k$ a.e.\ as $\varepsilon\to 0.$ 
	%
%
	By definition one has $ u_{k,\varepsilon}\leq u_k\leq u$ and recall that by assumption $u^{p-1}\nabla\phi\in L^{p'}(\omega;E)$;  thus, by the dominated convergence theorem we have that
	\[
	u_{k,\varepsilon}^{p-1}\nabla\phi\to u_{k}^{p-1}\nabla\phi \quad \text{in $L^{p'}(\omega;E)$ as $\varepsilon\to 0$}.
	\]
	One can also prove that $\nabla 	u_{k,\varepsilon}$ converges to $\nabla 	u_{k}$ in the sense of distributions and $\nabla 	u_{k,\varepsilon}$ is uniformly bounded in $L^p(E)$ w.r.t. $\varepsilon>0;$ hence, $\nabla 	u_{k,\varepsilon}$ converges weakly in $L^p(E)$ to $\nabla u_k$ as $\varepsilon\to 0.$
	Combining these facts, we obtain that
	\[
	\int_E u_{k,\varepsilon}^{p-1}\nabla u_{k,\varepsilon} \cdot \nabla \phi \omega dx \to \int_E u_k^{p-1}\nabla u_k\cdot \nabla \phi \omega dx \quad \text{as $\varepsilon\to 0$}.
	\]
	Taking into account Fatou's lemma and letting $\varepsilon\to 0$ in \eqref{eq:GaussGreen3} we get
	\begin{equation}\label{final-1}
		\int_E u_{k}^p\left(\omega \Delta_A\phi +\nabla \omega \cdot \nabla \phi\right)dx\leq-p\int_E u_{k}^{p-1}\nabla u_{k}\cdot \nabla \phi \omega dx.
	\end{equation}
	
	Now, it remains to take the limit $k\to \infty$ in \eqref{final-1}. To do this, we first observe that since $x\mapsto d(x,\partial E)$ is locally Lipschitz, it is differentiable a.e.\ in $E$; furthermore, for a.e.\ $x\in E$ one has that $\nabla d(\cdot,\partial E)(x)=-{\bf n}(x^*)$, where $x^*\in \partial E$ is the unique point with  $|x-x^*|=d(x,\partial E)$ and $\textbf{n}(x^*)$ is the unit outer normal vector at $x^*\in \partial E$. Since $E$ is convex and $\nabla\phi(x)\in \overline E$ for a.e.\ $x\in \overline E$, it turns out that 
	$\textbf{n}(x^*)\cdot \nabla\phi(x)\leq 0$ for a.e. $x\in E$. In particular,  the monotonicity of $\theta_k^0$ implies that for a.e. $x\in E$ 
	one has 
	\begin{eqnarray*}
		\nabla u_{k}(x)\cdot \nabla \phi(x)&=&\theta_k(x)\nabla u(x)\cdot \nabla \phi(x)-u(x)(\theta_k^0)'(d(x,\partial E))\textbf{n}(x^*)\cdot \nabla \phi(x)\\&\geq &\theta_k(x)\nabla u(x)\cdot \nabla \phi(x).
	\end{eqnarray*}
	The above arguments together with 		
	%
	\eqref{final-1} imply
	$$\int_E \theta_k^p u^p\left(\omega \Delta_A\phi +\nabla \omega \cdot \nabla \phi\right)dx\leq-p\int_E \theta_k^pu^{p-1}\nabla u\cdot \nabla \phi \omega dx.$$
	Passing to the limit  in the latter inequality as $k\to \infty$ and taking into account that $\theta_k\to 1$ for a.e. $x\in E$ as $k\to \infty$,  Fatou's lemma  and the dominated convergence theorem imply the desired inequality \eqref{eq:Byparts}. 
\end{proof}

We are ready to prove Theorem \ref{log-Sobolev}; we distinguish the cases when  $\tau>0$ and $\tau=0$, respectively.  
\subsection{Proof of the inequality \eqref{sharp-log-Sobolev}: case $p>1$, $\tau>0$}\label{subsection-3-1-1}

Fix  $u\in {W}^{1,p}(\omega;E)$ such that $\displaystyle\int_E |u|^p\omega dx=1.$ On account of relation  \eqref{ineq_2} by Proposition \ref{proposition-key},  we have  $\mathcal E_{\omega,E}(|u|^p)<+\infty$. 
 
  We first notice that when  $\mathcal E_{\omega,E}(|u|^p)=\displaystyle	\int_E |u|^p\log |u|^p \omega dx=-\infty$, we have nothing to prove in  \eqref{sharp-log-Sobolev}. 
 Accordingly, we may assume in the sequel that 
 \begin{equation}\label{entropy-1}
- \infty <\mathcal E_{\omega,E}(|u|^p)< +\infty.
 \end{equation}
 Using this relation, combined with relations \eqref{ineq_1}-\eqref{ineq_3} from Proposition \ref{proposition-key}, we can conclude that 
 \begin{equation}\label{J-functional}
 	-\infty<\int_E |u|^p\omega\log\omega dx<+\infty.
 \end{equation}
 
 We begin now  by introducing some  constants that will be used throughout the proof. 
 Let $\omega_{SE}\coloneqq \displaystyle\int_{\mathbb S^{n-1}\cap E}\omega$ and for further convenience, we introduce the real numbers $C_i=C_i(\omega,E,n,p,\tau)$, $i=1,...,4,$ as follows: $$C_1\coloneqq  p'\left(\omega_{SE}\Gamma\left(\frac{n+\tau}{p'}\right)\right)^{-1},$$
$$C_2\coloneqq \log C_1 -C_1\int_Ee^{-{|y|^{p'}}}|y|^{p'} \omega(y)dy=\log\left(p'\left(\omega_{SE}\Gamma\left(\frac{n+\tau}{p'}\right)\right)^{-1}\right)-\frac{n+\tau}{p'},$$
$$C_3\coloneqq  C_2+C_1\int_Ee^{-{|y|^{p'}}} \omega(y)\log \omega(y) dy,$$
and $$C_4\coloneqq  pC_1^\frac{1}{p'}\left(\int_Ee^{-{|y|^{p'}}}|y|^{p'} \omega(y)dy\right)^\frac{1}{p'}=p\left(\frac{n+\tau}{p'}\right)^\frac{1}{p'}.$$
We notice that the numbers $C_i\in \mathbb R$ are well-defined, $i=1,...,4.$

 We consider the scaled function $u_t(x)=t^\alpha u(t x)$ for $t>0$, $x\in E;$ by the homogeneity of $\omega$ we observe that $\displaystyle\int_E |u_t|^p\omega dx=1$ for every $t>0$ if and only if $\alpha p=n+\tau.$ By standard scaling arguments we also have that
 \begin{equation}\label{scaling-1}
 	\mathcal E_{\omega,E}(|u_t|^p)=\mathcal E_{\omega,E}(|u|^p)+\alpha p \log t,\ \ t>0,
 \end{equation}
 and 
 \[
 \int_E|\nabla u_t|^p\omega=t^p\int_E|\nabla u|^p\omega;
 \]
  therefore,  $u_t\in W^{1,p}(\omega;E)$ and inequality \eqref{sharp-log-Sobolev} holds for $u\in W^{1,p}(\omega;E)$ if and only if it holds for $u_t$. In particular, due to \eqref{J-functional}, we may choose $t>0$ such that
\begin{equation*}
\int_E|u_t(x)|^p\omega(x)\log\omega(x)dx=\int_E |u(x)|^p\omega(x)\log{\omega(x)}dx-\tau \log t=C_3-C_2.
\end{equation*}
 From now on we will simply write $u$ instead of $u_t$ and assume without loss of generality that
 \begin{equation}\label{scaling-vanish}
 \int_E|u(x)|^p\omega(x)\log\omega(x)dx=C_3-C_2.
 \end{equation}
 
 
 
Now, we shall focus to the proof of the log-Sobolev inequality \eqref{sharp-log-Sobolev}. Let us consider the probability measures $d\mu(x)=|u(x)|^p\omega(x)dx$ and $d\nu(x)=C_1e^{-|y|^{p'}}\omega(y)dy$. By the theory of optimal mass transport developed by Brenier \cite{Brenier} and further extended by McCann \cite[Main Theorem]{McCann} for general probability measures,  
there exists a convex function $\phi\colon \mathbb R^n \to \mathbb R$ such that $\nabla \phi$ pushes $\mu$ forward  to $\nu$, i.e.\ $\nabla \phi_\#\mu=\nu,$ and  $\nabla \phi(\overline E)
\subseteq \overline E$. In addition,  according to Alexandrov's classical result, $\nabla \phi$ is a.e.\ differentiable, and the push-forward relation $\nabla \phi_\#\mu=\nu$ implies  the validity of the  Monge-Amp\`ere equation
\begin{equation}\label{Monge-Ampere-log}
|u(x)|^p\omega(x)= C_1e^{-|\nabla\phi(x)|^{p'}}\omega(\nabla\phi(x)) \det(D_A^2\phi(x))\ \ {\rm a.e.}\ x\in E\cap U,
\end{equation}
where $U={\rm supp} u$ and  $D_A^2\phi$ denotes the Hessian of $\phi$ in the sense of Alexandrov (namely, the absolutely continuous part of the distributional Hessian of $\phi$), see McCann \cite[Remark 4.5]{McCann-Adv}; for further discussion, one may consult the monographs by Ambrosio, Bru\'e and Semola \cite[Lecture 5]{ABS}, Maggi \cite[Chapter 8]{Maggi} and Villani \cite[Section 3.3]{Villani}.  
 
In the sequel, we will use several times the change of variables via the  Monge-Amp\`ere equation \eqref{Monge-Ampere-log} which is  understood in the sense that $\nabla \phi_\#\mu=\nu$, equivalent to
\begin{equation}\label{change-of-variable}
	\int_E b(\nabla \phi(x))|u(x)|^p\omega(x)dx=C_1\int_{E} b(y)e^{-|y|^{p'}}\omega(y) dy
\end{equation}
for every Borel function $b:E\to \mathbb R.$

 

  By taking the logarithm of relation 
\eqref{Monge-Ampere-log}, and integrating with respect to the measure $|u(x)|^p\omega(x)dx$, we obtain that
\begin{eqnarray*}
	I&\coloneqq &\mathcal E_{\omega,E}(|u|^p)=\int_E |u(x)|^p\log(|u(x)|^p)\omega(x)dx\\&=&\log C_1-\int_E |\nabla\phi(x)|^{p'}|u(x)|^p\omega(x)dx+ \int_E |u(x)|^p\omega(x)\log{\omega(\nabla\phi(x))}dx\\&&-\int_E |u(x)|^p\omega(x)\log{\omega(x)}dx+\int_E |u(x)|^p\omega(x)\log\det(D_A^2\phi(x))dx.
\end{eqnarray*}
The change of variables formula \eqref{change-of-variable} applied to the second and third terms (for $b(y):=|y|^{p'}$ and $b(y):=\log \omega(y)$, respectively) gives
\[
I=C_3-\int_E |u(x)|^p\omega(x)\log{\omega(x)}dx+\int_E |u(x)|^p\omega(x)\log\det(D_A^2\phi(x))dx.
\]
\noindent On the other hand, by the weighted AM-GM inequality we have that
\begin{equation}\label{am-gm}
	\det(D_A^2\phi(x))^\frac{1}{n+\tau}\leq \frac{\tau}{n+\tau}+\frac{{\rm tr}D_A^2\phi(x)}{n+\tau}=\frac{\tau}{n+\tau}+\frac{\Delta_A \phi(x)}{n+\tau}\ \ {\rm for\ a.e.}\ x\in E\cap U.
\end{equation}
The latter estimate and Jensen's inequality imply that
\begin{eqnarray}\label{estimate-of-II}
	II&\coloneqq &\nonumber \int_E |u(x)|^p\omega(x)\log\det(D_A^2\phi(x))dx\\&\leq &\nonumber (n+\tau) \int_E |u(x)|^p\omega(x)\log \left(\frac{\tau}{n+\tau}+\frac{\Delta_A \phi(x)}{n+\tau}\right) dx\\&\leq &\nonumber (n+\tau) \log\left(\int_E |u(x)|^p\omega(x) \left({\tau}+{\Delta_A \phi(x)}\right) dx\right)-(n+\tau)\log(n+\tau)\\&= & (n+\tau) \log\left(\tau +\int_E |u(x)|^p\omega(x) {\Delta_A \phi(x)} dx\right)-(n+\tau)\log(n+\tau).
\end{eqnarray}
We are going to estimate the integral in the first term on the right side of the above expression. In order to do this, we first observe that by a change of variables we have that $u^{p-1}\nabla\phi\in L^{p'}(\omega;E)$. Accordingly, the integration by parts formula from Proposition \ref{prop:integrationbyparts} and H\"older's inequality imply that
\begin{eqnarray}\label{Definition-of-III}
 	III&\coloneqq &\nonumber\int_E |u(x)|^p\omega(x) {\Delta_A \phi(x)} dx\\&\leq  &\nonumber
 	-p\int_E|u(x)|^{p-1}\omega(x) {\nabla \phi(x)}\cdot\nabla |u(x)| dx-\int_E|u(x)|^{p}\nabla \omega(x) \cdot\nabla \phi(x) dx\\&\leq &\nonumber p\left(\int_E|\nabla u|^p\omega\right)^\frac{1}{p}\left(\int_E|u(x)|^p \omega(x)|\nabla \phi(x)|^{p'}\right)^\frac{1}{p'}-\int_E|u(x)|^{p}\nabla \omega(x) \cdot\nabla \phi(x) dx\\&=&C_4\left(\int_E|\nabla u|^p\omega\right)^\frac{1}{p}-\int_E|u(x)|^{p}\nabla \omega(x) \cdot\nabla \phi(x) dx,
 \end{eqnarray}
 where we used again the
change of variables formula \eqref{change-of-variable} and  the exact form of $C_4>0$, respectively.
 
  Since $\omega$ is log-concave on $E$, by the fact that $\nabla \phi(\overline E)
 \subseteq \overline E$ and Proposition \ref{prop-log-concave},   we have that
 \begin{equation}\label{log-concave}
 	\log\left(\frac{\omega(\nabla \phi(x))}{\omega(x)}\right)\leq -\tau+\frac{\nabla \omega (x)\cdot \nabla \phi(x)}{\omega (x)}\ \ {\rm for\ a.e.}\ x\in E.	
 \end{equation}
By relations \eqref{log-concave}, \eqref{Monge-Ampere-log}  and \eqref{scaling-vanish}, it follows that 
 \begin{eqnarray}\label{estimate-of-III}
 	III&\leq&\nonumber C_4\left(\int_E|\nabla u|^p\omega\right)^\frac{1}{p}-\int_E|u(x)|^{p}\nabla \omega(x) \cdot\nabla \phi(x) dx\\&\leq &C_4\left(\int_E|\nabla u|^p\omega\right)^\frac{1}{p}-\tau .
 \end{eqnarray}

Summing up relations \eqref{estimate-of-II} and \eqref{estimate-of-III} we conclude that 
\begin{eqnarray}\label{ineq-last-log}
\nonumber \int_E|u(x)|^p\log(|u(x)|^p)\omega(x)dx&\leq &C_3-\int_E |u(x)|^p\omega(x)\log{\omega(x)}dx-(n+\tau)\log(n+\tau)\\&&+(n+\tau) \log\left(C_4\left(\int_E|\nabla u|^p\omega\right)^{\frac{1}{p}} \right).
\end{eqnarray}
Taking relation \eqref{scaling-vanish} again into account, the inequality \eqref{ineq-last-log} reduces to 
\begin{eqnarray*}
	 \int_E|u(x)|^p\log(|u|^p)\omega(x)dx&\leq& C_2-(n+\tau)\log(n+\tau)+\frac{n+\tau}{p} \log\left(C_4^p\int_E|\nabla u|^p\omega\right) \\&=&
	 \frac{n+\tau}{p}	\log\left(\mathcal L_{\omega,p}\displaystyle\int_E |\nabla u|^p\omega dx\right),
\end{eqnarray*}
where by using the values of $C_2$ and $C_4$ and the fact that $\omega_{SE}=\displaystyle\int_{\mathbb S^{n-1}\cap E}\omega =(n+\tau)\int_{B\cap E}\omega$,  we obtain that
$$\mathcal L_{\omega,p}=\frac{e^\frac{pC_2}{n+\tau}}{(n+\tau)^p}C_4^p=\frac{p}{n+\tau}\left(\frac{p-1}{e}\right)^{p-1}\left(\Gamma\left(\frac{n+\tau}{p'}+1\right)\int_{B\cap E}\omega\right)^{-\frac{p}{n+\tau}},$$
which is precisely the expected constant, see \eqref{sharp-log-Sobolev}. \hfill $\square$

\begin{remark}\rm 
	The proof of \eqref{sharp-log-Sobolev} is substantially simpler for functions belonging to $C_c^\infty(\mathbb R^n)$ (restricted to $E$) than for functions in $W^{1,p}(\omega;E)$;  indeed, both Propositions \ref{proposition-key} and \ref{prop:integrationbyparts} as well as relation \eqref{J-functional} -- which are crucial in the proof -- trivially hold for functions in $C_c^\infty(\mathbb R^n)$. However, the $C_c^\infty(\mathbb R^n)$-version is not sufficient to characterize the equality cases in \eqref{sharp-log-Sobolev} as the candidates for the extremal functions do not belong to  $C_c^\infty(\mathbb R^n)$; see \S \ref{subsection-3-1-2}.
\end{remark}

\begin{remark}\rm \label{remark-anistropic}
	As we mentioned in the Introduction, it is possible to give an anisotropic version of Theorem \ref{log-Sobolev} w.r.t.\ a general norm $\|\cdot\|$. Indeed, if $\|\cdot\|_*$ is its dual, then in \eqref{sharp-log-Sobolev} the expression of the gradient norm is replaced by $\displaystyle \int_E \|\nabla u\|_*^p\omega dx$, while the ball $B$ in  the best constant $\mathcal L_{\omega,p}$ is understood w.r.t.\ the norm $\|\cdot\|$. In the same way, in the expression of the extremal function \eqref{Gaussian} the Euclidean norm is replaced by the norm  $\|\cdot\|$. The only modification of the proof is in the estimate \eqref{Definition-of-III}, where we use the inequality $X\cdot Y \leq \|X\|_* \|Y\|$ for every  $X,Y\in \mathbb R^n$, combined with the corresponding H\"older inequality; namely 
	\begin{eqnarray*}
		\int_E|u(x)|^{p-1} {\nabla u(x)}\cdot\nabla \phi(x) \omega(x)dx &\leq& \int_E|u(x)|^{p-1} \|\nabla u(x)\|_*\|\nabla \phi(x)\| \omega(x)dx\\ &\leq & \left(\int_E\|\nabla u\|_*^p\omega\right)^\frac{1}{p}\left(\int_E|u|^p \|\nabla \phi\|^{p'}\omega\right)^\frac{1}{p'}.
	\end{eqnarray*}
\end{remark}

\subsection{Equality in \eqref{sharp-log-Sobolev}: case $p>1$, $\tau>0$}\label{subsection-3-1-2}
We are going to  characterize the equality in \eqref{sharp-log-Sobolev}. 
 We assume that there exists a  function $u\in {W}^{1,p}(\omega;E)$ with equality in \eqref{sharp-log-Sobolev} and  $\displaystyle\int_E u^p\omega dx=1$; without loss of generality, we may assume that $u\geq 0$. 
Let $S$ be the interior of the effective domain of $\phi$, i.e.,   $\{x\in \mathbb R^n:\phi(x)<+\infty\}$. We observe that $U={\rm supp}u$ is contained in $\overline{E\cap S}$.  

As a first step we claim that for every compact subset $K\subset E\cap S$, there exist $c_K,C_K>0$ such that $c_K\leq u(x)\leq C_K $ for a.e. $x\in K.$ The proof of this property is inspired by \cite[Proposition 6]{CENV}. By tracking back the equality cases in the proof of \eqref{sharp-log-Sobolev}, we have -- among others -- that we should have equality in H\"older's inequality, see \eqref{Definition-of-III}; in particular, there exists $M>0$ such that 
 \[
 |\nabla u(x)|^p\omega(x)=Mu(x)^p|\nabla \phi(x)|^{p'}\omega(x)\ \ {\rm for\ a.e.}\ x\in E\cap S.
 \]
 For every $k\geq 1$, we consider the function $u_{k}(x)=\max\{\frac{1}{k},u(x)\}>0$, $x\in E\cap S$. Note  that $\nabla u_{k}\in L^p(\omega;E\cap S)$. In fact, we have for every $k\geq 1$ and a.e.\ $x\in E\cap S$ that
 \[
 |\nabla u_{k}(x)|^p\leq |\nabla u(x)|^p=Mu(x)^p|\nabla \phi(x)|^{p'}\leq  Mu_{k}(x)^p|\nabla \phi(x)|^{p'}.
 \]
  This estimate implies that for every $k\geq 1$ one has
 \[
 |\nabla\log  u_{k}(x)|^p \leq M|\nabla \phi(x)|^{p'}\ \ {\rm for\ a.e.}\ x\in E\cap S.
 \]
 Taking into account that $\phi$ is convex, $|\nabla\phi|$ is locally bounded on $E\cap S$. Thus, the functions $\log u_{k}$ are uniformly locally Lipschitz on $E\cap S$ with respect to $k;$ in particular, $\log u_{k}$ are also locally bounded on $E\cap S,$ i.e., there are $c_K^0,C_K^0\in \mathbb R$ such that 
 \[
 c_K^0\leq \log u_{k}(x)\leq C_K^0,\ \ \forall k\geq 1\ {\rm and\ for\ a.e.}\ x\in K.
 \]
  Letting $k\to \infty$, it turns out that 
\begin{equation}\label{two-sided estimate}
 c_K\leq u(x)\leq C_K\ {\rm  for\ a.e.}\ x\in K,
\end{equation}
where $c_K=e^{c_K^0}>0$ and $C_K=e^{C_K^0}>0$ which concludes the claim. In particular, we have that $U={\rm supp}u=\overline{E\cap S}$. 
 
 Similarly to  Cordero-Erausquin,  Nazaret and Villani \cite{CENV} (see also Nguyen \cite{Nguyen}), we prove that $\Delta_s \phi$ vanishes, where  $\Delta_s \phi$ stands for the singular part of the distributional Laplacian $\Delta_{\mathcal D'} \phi$. Note that $\Delta_s \phi$ is a non-negative measure and  $\Delta_{\mathcal D'} \phi=\Delta_A \phi+\Delta_s \phi$. Since we should have equality in \eqref{eq:GaussGreen}, repeating the approximation argument of Proposition \ref{prop:integrationbyparts} for the function $u$, we necessarily have that
 \begin{equation}\label{eq:limlim}
 \lim_{k\to \infty}\lim_{\varepsilon\to 0}\liminf_{\delta\to 0}\langle (u_{k,\varepsilon}^\delta)^p\omega, \Delta_s \phi\rangle_{\mathcal D'}=0.
 \end{equation}


 Fix a convex and compact set $K\subset E\cap S$ containing $\tilde x$ in its interior, and define $d_K\coloneqq  d(K, (E\cap S)^c)>0$. The set $K'\coloneqq \{x\in E\cap S: d(x,K)\leq d_K/2 \}$ is a convex compact subset of $E\cap S$ which contains $K$. Let $k\geq \frac{1}{d(K',\partial (E\cap S))}$  and  $\varepsilon>0$ be small enough such that $\frac K{1-\varepsilon}-\frac{\varepsilon}{1-\varepsilon}\tilde x\subseteq K'$ and $\varepsilon<\frac{1}{\max_{x\in K'}|x|}$. By using \eqref{two-sided estimate} and the fact that $\theta_k(x) =1$ for  $d(x,\partial E)\geq \frac{1}{k}$, it turns out that  $u_{k,\varepsilon}\geq c_{K'}$ a.e. on $K'$. In particular, if $0<\delta<d_K/2$, then by the definition of $K'$ we have $K+\delta B(0,1)\subset K'$ and consequently, a change of variables implies that for every $x\in K$, we have  
\begin{eqnarray*}
		u_{k,\varepsilon}^\delta(x)&=&\int_E u_{k,\varepsilon}(y)\kappa_\delta(x-y) dy=\frac{1}{\delta^n}\int_{|x-y|<\delta} u_{k,\varepsilon}(y)\kappa\left(\frac{x-y}{\delta}\right) dy\\&=&\int_{B(0,1)} u_{k,\varepsilon}(x-\delta z)\kappa\left(z\right) dz\\&\geq &c_{K'}\int_{B(0,1)} \kappa\left(z\right) dz\\&=&c_{K'}.
	\end{eqnarray*}
	Therefore, for such values of $\delta$ and $\varepsilon$ we have
	\[
	\langle (u_{k,\varepsilon}^\delta)^p\omega, \Delta_s \phi\rangle_{\mathcal D'}\geq c_{K'}^p\omega_{K}\Delta_s\phi[K]\geq0,
	\]
		where $\omega_{K}=\min_{K}\omega>0$.
	The latter estimate together with \eqref{eq:limlim} imply that $\Delta_s\phi[K]=0$. By the arbitrariness of $K$, $\Delta_s\phi$ should vanish. In particular, $\nabla \phi\in W^{1,1}(E)$.



The equality in \eqref{am-gm} implies that $D^2_A\phi(x)=I_n$ for a.e.\ $x\in E$. Since $\Delta_s\phi=0$ we can apply Figalli,  Maggi and  Pratelli \cite[Lemma A.2]{FMP} and find $x_0\in \overline E$ such that
\begin{equation}\label{nabla-phi=id}
	\nabla \phi(x)=x+x_0\ \ {\rm for\ a.e.}\ x\in E.
\end{equation}
In particular, $\overline E+x_0\subseteq \overline E.$
Moreover, the equality in \eqref{log-concave} with \eqref{nabla-phi=id} and the smoothness of $\omega$  implies that
\begin{equation}\label{log-konk-equal-2}
	\log\left(\frac{\omega(x+x_0)}{\omega(x)}\right)=\frac{\nabla \omega (x)\cdot x_0}{\omega (x)},\   x\in E.
\end{equation}
If we assume by contradiction that $x_0\in E$, then the latter relation with $x\coloneqq x_0$ gives that $\tau \log 2=\tau$, a contradiction to $\tau>0.$ Thus $x_0\in \partial E$. 

Moreover, if we consider for every $x\in E$ the function $r_x:E\to \mathbb R$ defined by
$$r_x(y)=\log\left(\frac{\omega(y)}{\omega(x)}\right)+\tau-\frac{\nabla \omega (x)\cdot y}{\omega (x)},\ y\in E,$$
then the log-convavity of $\omega$ on $E$ is equivalent to $r_x(y)\leq 0$ for every $y\in E$. Moreover,  \eqref{log-konk-equal-2} implies that $r_x(x+x_0)=0$, thus $y\coloneqq x+x_0\in E$ is a maximum point of $r_x$; in particular, $\nabla r_x|_{y=x+x_0}=0$, i.e., 
$$\frac{\nabla \omega(x+x_0)}{\omega(x+x_0)}-\frac{\nabla \omega (x)}{\omega (x)}=0,\ x\in E.$$
Therefore, there exists $c>0$ such that $\omega(x+x_0)=c\omega(x)$ for every $x\in E$. 

We are going to prove that $c=1$. First, by \eqref{log-konk-equal-2} we have that $\omega(x)\log c=\nabla \omega (x)\cdot x_0$ for every $x\in E.$ Differentiating $\omega(x+x_0)=c\omega(x)$, the homogeneity of $\omega$ implies that for every $x\in E$ we have
\begin{eqnarray}\label{constant-c}
\nonumber	c\tau\omega(x)&=&\tau\omega(x+x_0)=\nabla \omega(x+x_0)\cdot (x+x_0)=c\nabla \omega(x)\cdot (x+x_0)\\&=&c\tau\omega(x)+c\omega(x)\log c,
\end{eqnarray}
i.e., we necessarily have that $c=1.$
%
Thus, $\omega(x+x_0)=\omega(x)$ for every $x\in E$. In particular, by the Monge-Amp\`ere equation 
\eqref{Monge-Ampere-log} and \eqref{nabla-phi=id} 
one has that 
\[
u(x)^p= C_1e^{-|x+x_0|^{p'}} ,\ \ \ x\in E,
\]
i.e., the extremal function $u$ in \eqref{sharp-log-Sobolev}  is necessarily a Gaussian of the type \eqref{Gaussian} with the usual normalization property. 

We now prove that $x_0\in (-\partial E)\cap (\partial E)$. Notice that since $\nabla\phi(x)=x+x_0$ for a.e.\ $x\in E$, we can assume without loss of generality that $\nabla\phi(x)=x+x_0$ for every $x\in E$. Since $\nabla\phi$ is bijective (see \cite[Theorem~2.12(iv)]{Villani}), we have $E+x_0=E$. In particular, $x_0\in (-\overline E)\cap\overline E$. 

Conversely, if $u=	u_{\lambda,x_0}$ from \eqref{Gaussian}, by a change of variables and using the fact that $\omega(x+x_0)=\omega(x)$ for every $x\in E$, a simple computation shows that equality holds in \eqref{sharp-log-Sobolev}.  \hfill $\square$

\subsection{Proof of the inequality \eqref{sharp-log-Sobolev}: case $p>1$, $\tau=0$} Since $\tau=0$, Proposition \ref{prop-log-concave} implies that $\omega$ is  constant in $E$, i.e. $\omega(x)= \omega_0>0$ for every $x\in E$; moreover, without loss of generality, a simple computation shows that -- by choosing $v\coloneqq  u\omega_0^{1/p}$ in the original inequality \eqref{sharp-log-Sobolev} --  we may consider $\omega_0=1$.  Accordingly, the proof performed in \S\ref{subsection-3-1-1} 
becomes even simpler, e.g. the validity of \eqref{J-functional} is trivial. Furthermore, it turns out that
$C_2=C_3$ (as $\omega_0=1$)
which simplifies the scaling argument, see \eqref{scaling-vanish}, obtaining directly the inequality \eqref{sharp-log-Sobolev}. \hfill $\square$

\subsection{Equality in \eqref{sharp-log-Sobolev}: case $p>1$, $\tau=0$}  Analogously to \S \ref{subsection-3-1-2}, the equality in \eqref{sharp-log-Sobolev} implies that $\Delta_s\phi$  vanishes, thus $\nabla \phi\in W^{1,1}(E)$. 
The equality in \eqref{am-gm} (with $\tau=0$) implies that for some $c_0>0$ we have  $D^2\phi(x)=c_0I_n$ for a.e.\ $x\in E$. In particular, there exists  $x_0\in \overline E$ such that
\begin{equation}\label{nabla-phi=id-0}
	\nabla \phi(x)=c_0x+x_0\ \ {\rm for\ a.e.}\ x\in E.
\end{equation}
Thus, by the Monge-Amp\`ere equation 
\eqref{Monge-Ampere-log} and \eqref{nabla-phi=id-0}, since the weight is constant,  
we have that 
\[
u(x)^p= C_1c_0^ne^{-|c_0x+x_0|^{p'}} ,\ \ \ x\in E,
\]
i.e., the only class of extremal functions (up to scaling) is provided by the Gaussians \eqref{Gaussian}. 
Arguing as in the case $\tau>0$, we can assume without loss of generality that $\nabla\phi(x)=c_0x+x_0$ for every $x\in E$. Since $\nabla\phi$ is bijective, $c_0E+x_0=E$. Since $0\in\overline E$, we have that $-{x_0/{c_0}}\in \overline E$, whence $x_0 \in (-\overline E)\cap \overline E$, as required. \hfill $\square$

\section{Proof of Theorem \ref{log-Sobolev-1}:  the case $p=1$} \label{section-4}

The proof of Theorem \ref{log-Sobolev-1} requires the counterpart of Proposition \ref{prop:integrationbyparts} in the case $p=1$, i.e., an integration by parts formula/inequality for functions belonging to $BV(\omega;E)$. 

\begin{proposition}\label{prop:p=1}
	Let $E\subseteq\mathbb R^n$ be an open convex cone, $u\in BV(\omega;E)$ be non-negative and let $\phi\colon \mathbb R^n\to \mathbb R$ be a convex function such that $|\nabla\phi(x)|\leq 1$ and $\nabla\phi(x)\in E$ for a.e. $x\in E$.
	Then we have
	\begin{equation}\label{eq:Bypartsp=1}
		\int_E u\omega \Delta_A\phi dx\leq  \|Du\|_{\omega}(E) -\int_E u \nabla \omega \cdot \nabla \phi dx.
	\end{equation}
\end{proposition}
\begin{proof}
	If $u\in L^1(\omega;E)\cap C^\infty_c(\mathbb R^n)$, then the divergence theorem yields
	\begin{eqnarray}\label{eq:intbypartssmooth}
		\int_E u\omega \Delta_A\phi dx &\leq& \int_E u\omega \Delta_{\mathcal D'}\phi dx \\&=& -\int_E \nabla\phi\cdot \nabla u \omega dx -\int_E u\nabla\omega\cdot \nabla\phi dx +\int_{\partial E} u\omega \widetilde{\nabla\phi}\cdot{\bf n} d\mathcal H^{n-1},
	\end{eqnarray}
	where $\widetilde{\nabla \phi}$ denotes the trace of $\nabla\phi$ on $\partial E$, which is well-defined up to $\mathcal H^{n-1}$-null sets. One can verify that since $\nabla\phi(x)\in E$ for a.e.\  $x\in \overline E$, then $\widetilde{\nabla\phi}\cdot {\bf n}\leq 0$ holds $\mathcal H^{n-1}$-a.e. on $\partial E$, whence
	\begin{equation}\label{eq:negativeboundary}
		\int_{\partial E} u\omega \widetilde{\nabla\phi}\cdot {\bf n} d\mathcal H^{n-1}\leq 0.
	\end{equation}
	On the other hand, since $|\nabla \phi|\leq 1$ a.e.\ on $E$, one has
	\begin{equation}\label{eq:totvariation}
		-\int_E\nabla \phi \cdot \nabla u\omega dx\leq \|Du\|_{\omega}(E).
	\end{equation}
	Combining \eqref{eq:negativeboundary}, \eqref{eq:totvariation} and \eqref{eq:intbypartssmooth} we obtain  \eqref{eq:Bypartsp=1} for $u\in L^1(\omega;E) \cap  C_c^\infty(\mathbb R^n)$.
	
	Assume now $u\in BV(\omega;E)$. If we have a family of $u_\varepsilon \in C^\infty(E)\cap BV(\omega;E)$ such that
	\begin{equation}\label{approximation-0}
	\begin{aligned}
		&\lim_{\varepsilon\to0}\int_E |u_\varepsilon-u|\omega dx=0,\\
		&\lim_{\varepsilon\to 0} \|Du_\varepsilon\|_{\omega}(E)=\|Du\|_\omega(E),
	\end{aligned}
\end{equation}
	the proof is concluded. Indeed,  by the previous part we have that	\begin{equation}\label{eq:u_h}
		\int_E u_\varepsilon\omega \Delta_{A}\phi dx+\int_E u_\varepsilon\nabla\omega\cdot \nabla\phi dx\leq \|Du_\varepsilon\|_\omega(E), 
	\end{equation}
	and taking the limit of \eqref {eq:u_h} as $\varepsilon\to 0$, the approximations from \eqref{approximation-0} together with Fatou's lemma imply the required inequality \eqref{eq:Bypartsp=1}.
	
	In the sequel we shall prove \eqref{approximation-0} on a generic open set $\Omega\subseteq E$, i.e., we provide  	
	the existence of $u_\varepsilon \in C^\infty(\Omega)\cap BV(\omega;\Omega)$, $\varepsilon\in (0,1)$, such that the approximation properties 
	\begin{equation}\label{approximation}
		\begin{aligned}
			&\lim_{\varepsilon\to0}\int_\Omega |u_\varepsilon-u|\omega dx=0,\\
			&\lim_{\varepsilon\to 0} \|Du_\varepsilon\|_{\omega}(\Omega)=\|Du\|_\omega(\Omega)
		\end{aligned}
	\end{equation}
	 hold. In fact,  this property is  well-known  by Bellettini,  Bouchitt\'e and Fragal\`a \cite[Theorem 5.1]{BBF}  for general finite Radon measures on $\mathbb R^n$ and one can extend by a standard diagonal argument to $\sigma$-finite measures.
		We however need to have an explicit expression for the approximating sequence, hence we give a direct proof of the result by adapting essentially the argument of Evans and  Gariepy  \cite[Theorem 5.3]{EG}.


%
To do this, fix $\varepsilon>0$. For any $m,k\in \mathbb N$, we define
	\[
	U_k=\left\{x\in \Omega: d(x,\partial \Omega)>\frac1{m+k}\right\}\cap B(0,m+k),
	\]
	and fix $m\in \mathbb N$ so large that $\|Du\|_\omega(\Omega\setminus U_1)<\varepsilon$. Observe that $U_k$ has compact closure and $d(\overline U_k,\partial \Omega)>0$. By setting $U_0\coloneqq  \emptyset$, we introduce the open sets $V_k$ by letting
	\[
	V_k\coloneqq  U_{k+1}\setminus\overline U_{k-1}, \quad k\geq 1
	.\]
	Clearly, $\bigcup_{k\geq 1}V_k=\Omega$. According to the partition of unity we can find a family $\zeta_k\in C_c^\infty(V_k)$ with $0\leq\zeta_k\leq 1$ and
	$
	\sum_{k=1}^\infty\zeta_k\equiv 1.
	$
	By the construction of $V_k$'s, the sum above evaluated at any point has only at most three non-zero elements.
	Consider a symmetric mollifier $\eta\in C_c^\infty(B(0,1))$ such that $\int\eta dx=1$ and define $\eta_\delta(x)\coloneqq  \delta^{-n}\eta( x/\delta)$. Denote the convolution between the kernel $\eta_\delta$ and  a function $f\in L^1_{\rm loc}(\Omega)$ with
	\[
	\eta_\delta\star f(x)=\int \eta_\delta (x-y)f(y)dy.
	\]
	It is classical that  if $f\in L^p(\Omega)$, then $\eta_\delta\star f\in L^p(\Omega)$ and 
	\begin{equation}\label{eq:Lpconv}
		\lim_{\delta\to 0} \int_\Omega |(\eta_\delta\star f)-f|^pdx=0.
	\end{equation}
	Moreover $\eta_\delta\star f\in C^\infty$ and if $f\in \mathcal C^1$, one has
	\begin{equation}\label{eq:convderivative}
		\partial_j(\eta_\delta\star f)=\eta_\delta\star (\partial_jf).
	\end{equation}
	For any $k\in \mathbb N$, the function $u\zeta_k$ is supported in $V_k$, whose closure is compact, hence $u\zeta_k\in L^1(\Omega)$. This follows by the estimate
	\[
	\int_\Omega |u|\zeta_k dx=\int_{V_k}|u|\zeta_kdx \leq \frac{1}{\min_{V_k}\omega}\int_{V_k}|u|\omega dx<+\infty.
	\]
	Analogously, $u\nabla\zeta_k$ is compactly supported in $V_k$ and $u\nabla\zeta_k\in L^1(\Omega)$. Fix $X\in C_c^1(\Omega;\mathbb R^n)$ such that $|X|\leq 1$. By \eqref{eq:Lpconv} and the fact that $V_k$ is compact and $\omega$ of class $\mathcal C^1$, we can choose $\varepsilon_k>0$ such that the following conditions hold:
	\begin{equation}\label{eq:3conditions}
		\begin{cases}
			{\rm supp}(\eta_{\varepsilon_k}\star (u\zeta_k))\subset V_k\\
			\displaystyle\int_\Omega |(\eta_{\varepsilon_k}\star (u\zeta_k)-u\zeta_k|\omega dx< \frac{\varepsilon}{2^k}\\
			\displaystyle\int_\Omega |(\eta_{\varepsilon_k}\star (u\nabla\zeta_k)-u\nabla\zeta_k|\omega dx< \frac{\varepsilon}{2^k}.
		\end{cases}
	\end{equation}
	Since $\omega X$ is uniformly continuous on $V_k$, we can moreover assume that $\varepsilon_k$ is small enough to ensure
	\begin{equation}\label{eq:stimacampo}
		|\zeta_k(\eta_{\varepsilon_k}\star \omega X)|(x)\leq\omega(x)|X|(x)+\varepsilon\min_{V_k}\omega\leq \omega(x)(1+\varepsilon)  \quad {\rm for\ all}\ x\in \Omega.
	\end{equation}
	We then define
	\[
	u_\varepsilon\coloneqq  \sum_{k=1}^\infty \eta_{\varepsilon_k}\star (u\zeta_k).
	\]
	Then $u_\varepsilon\in C^{\infty}(\Omega)$ and by \eqref{eq:3conditions} and the fact that $u=\sum_{k=1}^\infty u\zeta_k$ it holds that
	\[
	\lim_{\varepsilon\to0}\int_\Omega |u_\varepsilon-u|\omega dx=0.
	\] 
	By the lower semicontinuity of the total variation with respect to $L^1(\omega;\Omega)$-convergence we also have
	\[
	\|Du\|_\omega(\Omega)\leq \liminf_{\varepsilon\to 0}\|Du_\varepsilon\|_\omega(\Omega).
	\]
	To prove the opposite inequality we proceed as follows:
	\begin{equation}\label{eq:stimalunga}
		\begin{aligned}
			\int_\Omega u_\varepsilon {\rm div}(\omega X) dx&=\sum_{k=1}^\infty\int_\Omega \eta_{\varepsilon_k}\star (u\zeta_k){\rm div}(\omega X) dx=\sum_{k=1}^\infty\int_\Omega  (u\zeta_k)(\eta_{\varepsilon_k}\star{\rm div}(\omega X))dx\\
			&=\sum_{k=1}^\infty\int_\Omega  (u\zeta_k){\rm div}(\eta_{\varepsilon_k}\star\omega X)dx\\
			&=\sum_{k=1}^\infty\int_\Omega  u{\rm div}(\zeta_k(\eta_{\varepsilon_k}\star\omega X))dx-\sum_{k=1}^\infty\int_\Omega u\nabla\zeta_k \cdot (\eta_{\varepsilon_k}\star \omega X) dx\\
			&=\sum_{k=1}^\infty\int_\Omega  u{\rm div}(\zeta_k(\eta_{\varepsilon_k}\star\omega X))dx-\sum_{k=1}^\infty\int_\Omega X\cdot(\eta_{\varepsilon_k}\star u\nabla\zeta_k) \omega dx.
		\end{aligned}
	\end{equation}
	Here, we have used \eqref{eq:convderivative} and the fact that, if $f\in L^1_{\rm loc}(\Omega)$  and $g \in L^1_{\rm loc}(\Omega)$ with $({\rm supp}(g))_\delta\subset \Omega$, then
	\[
	\int_\Omega (\eta_\delta\star f)gdx =\int_\Omega f(\eta_\delta\star g)dx.
	\]
	Taking into account that $\sum_{k=1}^\infty \nabla\zeta_k\equiv0$, the second sum in \eqref{eq:stimalunga} is
	\[
	I_2^\varepsilon\coloneqq -\sum_{k=1}^\infty\int_\Omega X\cdot(\eta_{\varepsilon_k}\star u\nabla\zeta_k) \omega dx=-\sum_{k=1}^\infty\int_\Omega X\cdot[(\eta_{\varepsilon_k}\star u\nabla\zeta_k) -u\nabla\zeta_k]\omega dx.
	\]
	By \eqref{eq:3conditions}, one has $|I_2^\varepsilon|< \varepsilon$. The first sum of \eqref{eq:stimalunga} can  be written as
	\[
	I_1^\varepsilon\coloneqq \sum_{k=1}^\infty\int_\Omega  u{\rm div}(\zeta_k(\eta_{\varepsilon_k}\star\omega X))dx=\int_\Omega u {\rm div}(\zeta_1(\eta_{\varepsilon_1}\star \omega X))dx+\sum_{k=2}^\infty\int_\Omega  u{\rm div}(\zeta_k(\eta_{\varepsilon_k}\star\omega X))dx.
	\]
	Taking into account \eqref{eq:stimacampo} and that $\zeta_k(\eta_{\varepsilon_k}\star \omega X)$ is compactly supported on $V_k$ we can write
	\[
	\begin{aligned}
		|I_1^\varepsilon|&\leq (1+\varepsilon)\|Du\|_\omega(\Omega)+ (1+\varepsilon)\sum_{k=2}^\infty \|Du\|_\omega(V_k)\\&\leq (1+\varepsilon)\|Du\|_\omega(\Omega)+ 3 (1+\varepsilon)\|Du\|_\omega(\Omega\setminus U_1)\\&\leq (1+\varepsilon)\|Du\|_\omega(\Omega)+ 3\varepsilon(1+\varepsilon).
	\end{aligned}
	\]
	Therefore, we  proved that
	\[
	\int_\Omega u_\varepsilon {\rm div}(\omega X)dx \leq (1+\varepsilon)\|Du\|_\omega(\Omega)+3\varepsilon(1+\varepsilon)+\varepsilon.
	\]
	The proof is concluded by the arbitrariness of $X\in C^1_c(\Omega;\mathbb R^n)$ with $|X|\leq1$, by taking the limit   $\varepsilon\to 0$. 
\end{proof}

\begin{remark}\label{rem:uniformzeta}\rm
	In the proof of Proposition \ref{prop:p=1}  is not restrictive to choose $\varepsilon_k$ also such that
	\[
	|(\eta_{\varepsilon_k}\star \zeta_k)(x)-\zeta_k(x)|<\frac{\varepsilon}{2^k} \quad {\rm for\ all}\  x\in \Omega.
	\]
	This is a consequence of the fact that, if $f$ is uniformly continuous, then the convolution $\eta_\delta \star f$ uniformly converges to $f$  as $\delta\to 0$.
\end{remark}

\subsection{Proof of the inequality \eqref{sharp-log-Sobolev-1}: case $p=1$, $\tau>0$}\label{subsection-4-1}
	The proof is similar to the one presented in \S\ref{subsection-3-1-1}; we only focus on the differences. 
Let $u\in BV(\omega;E)$ be such that $\displaystyle\int_E |u|\omega dx=1$. 

 If $u_t(x)=t^\alpha u(\lambda x)$ for $t>0$, $x\in E,$ the homogeneity of $\omega$ implies that $\displaystyle\int_E |u_t|\omega dx=1$ for every $t>0$ whenever $\alpha =n+\tau$. 
 Moreover, since 
 \[\mathcal E_{\omega,E}(|u_t|)=\mathcal E_{\omega,E}(|u|)+\alpha  \log t\ \ {\rm and}\ \  \|D(|u_t|)\|_{\omega}(E)=t\|D(|u|)\|_{\omega}(E), \quad t>0,
 \]
 one has that $u_t\in BV(\omega;E)$ and inequality \eqref{sharp-log-Sobolev-1} holds for some $u\in BV(\omega;E)$ if and only if it holds for $u_t$. Therefore may choose $t>0$ such that
 \[
  \int_E |u_t(x)|\omega(x)\log{\omega(x)}dx= \int_E |u(x)|\omega(x)\log{\omega(x)}dx-\tau \log t=\tilde C_1 \int_{B\cap E} \omega(x)\log \omega(x)dx,
  \]
  where $\tilde C_1\coloneqq \left(\displaystyle\int_{B\cap E}\omega\right)^{-1}$.
Without loss of generality, we will write $u$ in place of $u_t$ and assume
\begin{equation}\label{eq:scaling}
\int_E |u(x)|\omega(x)\log{\omega(x)}dx=\tilde C_1 \int_{B\cap E} \omega(x)\log \omega(x)dx.
\end{equation}
 

Let us consider the probability measures $|u(x)|\omega(x)dx$ and $\tilde C_1\mathbbm{1}_{B\cap E}(y)\omega(y)dy$. By the theory of optimal mass transport,   
we find a convex function $\phi\colon \mathbb R^n \to \mathbb R$ such that $\nabla \phi(\overline E)\subseteq \overline E$ and the Monge-Amp\`ere equation holds, i.e., 
\begin{equation}\label{Monge-Ampere-log-1}
	|u(x)|\omega(x)= \tilde C_1\mathbbm{1}_{B\cap E}(\nabla\phi(x))\omega(\nabla\phi(x)) \det(D_A^2\phi(x))\ \ {\rm for\ a.e.}\ x\in E\cap U,
\end{equation}
where $U$ is the support of $u$. In particular, by \eqref{Monge-Ampere-log-1} one has that 
$$|\nabla\phi(x)|\leq 1\ \ {\rm for\ a.e.}\ x\in E\cap U.$$
The latter bound for $\nabla \phi$ and the integration by parts formula from Proposition \ref{prop:p=1}  imply that
\begin{eqnarray}\label{Definition-of-I-1}
	\tilde I&\coloneqq &\nonumber\int_E |u(x)|\omega(x) {\Delta_A \phi(x)}dx\\&\leq&\|D(|u|)\|_{\omega}(E)-\int_E|u(x)|\nabla \omega(x) \cdot\nabla \phi(x) dx. 
\end{eqnarray}
By the log-concavity of $\omega$, see \eqref{log-concave}, similarly to \eqref{estimate-of-III}, the latter inequality implies  that
\begin{eqnarray}\label{estimate-of-I-1}
	\tilde I&\leq&\nonumber \|D(|u|)\|_{\omega}(E)-\int_E|u(x)|\nabla \omega(x) \cdot\nabla \phi(x) dx\\&\leq &\|D(|u|)\|_{\omega}(E)-\tau +\int_E |u(x)|\omega(x)\log{\omega(x)}dx-\tilde C_1 \int_{B\cap E} \omega(x)\log \omega(x)dx .
\end{eqnarray}


A similar argument as above, combined with the  equation \eqref{Monge-Ampere-log-1} shows that 
\begin{eqnarray*}
\widetilde	{II}&\coloneqq &\mathcal E_{\omega,E}(|u|)\\&=&\int_E |u(x)|\log(|u(x)|)\omega(x)dx\\&=&\log \tilde C_1-\int_E |u(x)|\omega(x)\log \omega(x)dx+ \tilde C_1 \displaystyle\int_{B\cap E} \omega(x)\log \omega(x)dx\\&&+\int_E |u(x)|\omega(x)\log\det(D_A^2\phi(x))dx.
\end{eqnarray*}
This relation together with the AM-GM inequality \eqref{am-gm}, the Jensen inequality and estimate \eqref{estimate-of-I-1} imply that 
\begin{eqnarray*}
	\widetilde	{II}&\leq &\log \tilde C_1-\int_E |u(x)|\omega(x)\log \omega(x)dx+ \tilde C_1 \displaystyle\int_{B\cap E} \omega(x)\log \omega(x)dx-(n+\tau)\log(n+\tau)\\&&+(n+\tau)\log\left(\|D(|u|)\|_{\omega}(E) +\int_E |u(x)|\omega(x)\log{\omega(x)}dx-\tilde C_1 \int_{B\cap E} \omega(x)\log \omega(x)dx \right).
\end{eqnarray*}
By \eqref{eq:scaling}, the latter estimate reduces to 
\[
\mathcal E_{\omega,E}(|u|)\leq \log \tilde C_1-(n+\tau)\log(n+\tau)+(n+\tau)\log\left(\|D(|u|)\|_{\omega}(E)  \right).
\]
We observe that this inequality is equivalent to 
\[\mathcal E_{\omega,E}(|u|)\leq (n+\tau)\log\left(\frac{\tilde C_1^{\frac{1}{n+\tau}}}{n+\tau}\|D(|u|)\|_{\omega}(E)  \right),\]
which is  exactly the required relation \eqref{sharp-log-Sobolev-1}. \hfill $\square$

\subsection{Equality in \eqref{sharp-log-Sobolev-1}: case $p=1$, $\tau>0$}\label{subsection-4-2} By the Monge-Amp\`ere equation \eqref{Monge-Ampere-log-1} and $\nabla \phi(\overline E)\subseteq \overline E$  we obtain that, up to a null-measure set, $\Omega\coloneqq  {\rm int}(E\cap U)$ coincides with the set $\{x\in E:|\nabla \phi(x)|< 1\}.$ In particular, $u(x)=0$ for a.e.\ $x\in E\setminus \Omega $.  On the other hand, we are going to prove that
\begin{equation}\label{u-estimate}
	|u(x)|\geq \tilde C_1 e^{-\tau}\ \ {\rm for \ a.e.}\ x\in \Omega.
\end{equation}
To do this, the equality in \eqref{sharp-log-Sobolev-1} implies that equalities should occur  both in the AM-GM inequality \eqref{am-gm} and \eqref{log-concave}; namely, $D^2_A\phi(x)=I_n$ for a.e.\ $x\in E$ and  
\begin{equation}\label{log-concave-1}
	\log\left(\frac{\omega(\nabla \phi(x))}{\omega(x)}\right)= -\tau+\frac{\nabla \omega (x)\cdot \nabla \phi(x)}{\omega (x)}\ \ {\rm for\ a.e.}\ x\in E.	
\end{equation}
In particular, the former relation implies that det$(D^2_A\phi(x))=1$ for a.e.\ $x\in E$, thus   the Monge-Amp\`ere equation \eqref{Monge-Ampere-log-1} reduces to 
\begin{equation}\label{Monge-Ampere-log-2}
	|u(x)|\omega(x)= \tilde C_1\omega(\nabla\phi(x)) \ \ {\rm for\ a.e.}\ x\in \Omega.
\end{equation}
Moreover, since $\nabla \phi(\overline E)\subseteq \overline E$, Proposition \ref{prop-log-concave} implies that $\nabla \omega (x)\cdot \nabla \phi(x)\geq 0$ for a.e. $x\in E$; thus, combining this inequality with \eqref{log-concave-1} and \eqref{Monge-Ampere-log-2}, the requested inequality \eqref{u-estimate} immediately follows.  

In the sequel, we prove that $\Delta_s\phi[K]=0$ for every compact set $K\subset \Omega.$ The equality in \eqref{sharp-log-Sobolev-1} also implies that 
 \begin{equation}\label{eq:limlim-1}
	\lim_{\varepsilon \to0}\left\langle |u|_{\varepsilon}\omega, \Delta_s \phi\right\rangle_{\mathcal D'}=0,
\end{equation}
see inequality \eqref{eq:Bypartsp=1}, where $(|u|)_\varepsilon\in BV(\omega;E)\cap C_c^{\infty}(\Omega)$ is the sequence approximating $|u|$ described in the proof of Proposition \ref{prop:p=1}, namely
\[
|u|_\varepsilon=\sum_{k=1}^\infty \eta_{\epsilon_k} \star (|u|\zeta_k),
\]
where $\zeta_k\in C_c^\infty(V_k)$ are relative to a partition of unity associated to the family of open sets $V_k$ and, according to Remark \ref{rem:uniformzeta}, $\varepsilon_k$ is such that
\begin{equation}\label{eta-needed}
	|(\eta_{\varepsilon_k}\star \zeta_k)(x)-\zeta_k(x)|< \frac{\varepsilon}{2^k}, \quad \forall x\in \Omega.
\end{equation} 
Let  $K\subset \Omega$ be a compact set; we claim that 
\begin{equation}\label{eq:stimadalbasso}
|u|_\varepsilon(x)\geq \frac {\tilde C_1}2 e^{-\tau}, \quad \forall x\in K, \forall \varepsilon \in \left(0,1/2\right).
\end{equation}
In fact, by \eqref{eta-needed}, for every $x\in K$ we have
\[
\begin{aligned}
|u|_\varepsilon(x)&=\sum_{k=1}^\infty \eta_{\epsilon_k} \star (|u|\zeta_k)=\sum_{k=1}^\infty\int_{\Omega} \eta_{\varepsilon_k}(x-y) |u|(y)\zeta_k(y)dy\\
& \geq \tilde C_1e^{-\tau}\sum_{k=1}^\infty(\eta_{\varepsilon_k}\star \zeta_k)(x)\geq \tilde C_1e^{-\tau}\sum_{k=1}^\infty\left(\zeta_k(x)-\frac{\varepsilon}{2^k}\right)\\
&=\tilde C_1e^{-\tau}(1-\varepsilon).
\end{aligned}
\]
Finally, since $\Omega\subset E$, we have that $K\cap \partial E=\emptyset$ and $\omega_K\coloneqq \min_K \omega>0$. 
 Therefore, by \eqref{eq:limlim-1} and \eqref{eq:stimadalbasso}, it yields that
\[
0=	\lim_{\varepsilon\to0}\left\langle |u|_{\varepsilon}\omega, \Delta_s \phi\right\rangle_{\mathcal D'}\geq \frac{\tilde C_1}{2} e^{-\tau}\omega_{K}\Delta_s\phi[K].
\]
Note that $\Delta_s\phi[K]\geq 0$; thus, we necessarily have that 
$\Delta_s\phi[K]=0$, which ends the proof. 

In particular, we have proved that $\phi\in W^{2,1}(\Omega)$. We aim to prove that $u$ is constant a.e.\ on $\Omega$.
Since $D^2\phi=I_n$ a.e.\ on $E$, there exist  a family of connected open sets $\Omega_i\subset E$ and points $x_0^i\in \mathbb R^n$, 
 $i\in I$, such that $\Omega=\bigcup_{i\in I} \Omega_i$ and
\[
\nabla\phi(x)=x+x_0^i \quad  \text{ for a.e. $x\in \Omega_i$}.
\]
Notice that by the Monge-Amp\`ere equation, up to a null-measure set, $\Omega_i\subseteq \{x\in E: x+x_0^i\in E, |x+x_0^i|<1 \}$, so that
\[
\Omega_i\subseteq (B(0,1)\cap E-x_0^i)\cap E.
\]
In particular $\mathcal L^n(\partial \Omega)=0$.
We can now localize \eqref{log-concave-1} on $\Omega_i$ and write
\begin{equation}\label{eq:log-concave-local}
\log\left( \frac{\omega(x+x_0^i)}{\omega(x)}\right)=-\tau+\frac{\nabla\omega(x)\cdot(x+x_0^i)}{\omega(x)} \quad \text{for a.e.\ $x\in \Omega_i$}.
\end{equation}
Arguing as in the case $p>1$, we can find $c_i>0$ such that
\begin{equation}\label{eq:proportionality}
	\omega(x+x_0^i)=c_i\omega(x), \quad \forall x\in \Omega_i;
\end{equation}
in fact, it turns out that $c_i=1$ for every $i\in I.$ 
Thus, the Monge-Ampère equation \eqref{Monge-Ampere-log-1} implies that for every $i\in I$, 
\begin{equation}\label{u-lambda-components}
	|u|(x)=\tilde C_1 \mathbbm 1_{B\cap E} (x+x_0^i)\quad  \text{for a.e. $x\in \Omega_i$}.
\end{equation}
We now exploit the equality in \eqref{eq:Bypartsp=1}. In particular, we have
\[
n\int_\Omega \tilde C_1 \omega(x)dx+\sum_{i\in I}\int_{\Omega_i} \tilde C_1\nabla\omega(x)\cdot (x+x_0^i)dx=\tilde C_1  \sum_{i\in I}P_{\omega}(\Omega_i;E),
\]
where $P_{\omega}(\Omega;E)$ stands for the weighted perimeter of $\Omega$ relative to the cone $E$; see e.g.  Cabr\'e,  Ros-Oton and Serra
\cite[p. 2977]{Cabre-Ros-Oton-Serra}; indeed, one has that $\|D\mathbbm 1_{\Omega}\|_\omega(E)=P_{\omega}(\Omega;E)$. Taking into account that $\nabla \omega(x)\cdot x_0^i=0$ for every $x\in E$ and $i\in I$ (see \eqref{eq:log-concave-local}), the latter relation implies that
\begin{equation}\label{eq:Isoperimetric}
(n+\tau)\int_\Omega \omega(x)dx=\sum_{i\in I}P_{\omega}(\Omega_i;E).
\end{equation}

We now prove that $\Omega$ has just one connected component. 
Since $\Omega_i\cap \Omega_j=\emptyset$ for $i\neq j$, then 
$$\sum_{i\in I}P_{\omega}(\Omega_i;E)=P_{\omega}(\Omega;E)$$ and -- since the optimal mass transport map $\nabla\phi$ is essentially one-to-one (see \cite[Theorem 2.12(iv)]{Villani}) -- we have that
\begin{equation}\label{eq:lebesguenull}
\mathcal L^n((\Omega_i+x_0^i)\cap (\Omega_j+x_0^j))=0,\ i\neq j.
\end{equation}
In particular, up to null-measure sets, we have that
$
\bigcup_{i\in I}(\Omega_i+x_0^i)=B\cap E.
$
 By \eqref{eq:proportionality} with $c_i=1$ we obtain   that
\[
\int_{\Omega_i}\omega(x) dx=\int_{\Omega_i}\omega(x+x_0^i)dx=\int_{\Omega_i+x_0^i}\omega(x)dx,\ \ i\in I.
\]
Combining the previous fact with \eqref{eq:lebesguenull}, it follows that 
\begin{equation}\label{eq:samevolume}
\int_\Omega\omega(x)dx=\int_{B\cap E} \omega(x)dx.
\end{equation}
Consequently, by using \eqref{eq:samevolume}, the equation   \eqref{eq:Isoperimetric} can be written into the equivalent form 
$$(n+\tau)\left(\int_{B\cap E} \omega(x)dx\right)^{\frac{1}{n+\tau}}\left(\int_\Omega \omega(x)dx\right)^{1-\frac{1}{n+\tau}}=P_{\omega}(\Omega;E).$$ 
The latter relation means that we have equality in the weighted isoperimetric inequality, and since $\omega^\frac{1}{\tau}$ is concave (see Proposition \ref{prop-log-concave}), we are within the setting of  Cinti,  Glaudo,  Pratelli,  Ros-Oton and Serra \cite{CGPROS}; namely, $\Omega$ has the form $\Omega=(B\cap E)-x_0$ for some $x_0\in \partial E\cap (-\partial E)$. In particular, $\Omega$ has only one connected component. 

An alternative, self-contained proof of this fact can be also obtained by observing -- similarly as above by $c_i=1$ and relation \eqref{eq:proportionality}  -- that  
\begin{equation}\label{eq:sameperimeter}
P_\omega(\Omega_i;E)=P_\omega(\Omega_i+x_0^i;E),\ \ i\in I.
\end{equation}
By a radial change of coordinates we also have
\begin{equation}\label{eq:volumeperimeter}
(n+\tau)\int_{B(0,1)\cap E}\omega(x)dx=\int_{\mathbb S^{n-1}\cap E}\omega d\mathcal H^{n-1}=P_\omega(B\cap E; E).
\end{equation}
Combining \eqref{eq:volumeperimeter}, \eqref{eq:sameperimeter} and \eqref{eq:samevolume} with \eqref{eq:Isoperimetric} we get
\begin{equation}\label{eq:summingup}
\begin{aligned}
\sum_{i\in I} P_\omega(\Omega_i+x_0^i;E)&=\sum_{i\in I}P_\omega(\Omega_i;E)=P_\omega(\Omega;E)=(n+\tau)\int_\Omega \omega dx\\
&=(n+\tau)\int_{B\cap E}\omega dx=P_\omega(B\cap E; E). 
\end{aligned}
\end{equation} 
If, by contradiction we would have $\#I\geq 2$, then the perimeter of $B\cap E$ would equal the sum of the perimeters of some of its open subsets, contradicting the connectedness of $B\cap E$. Therefore, there is $x_0\in \mathbb R^n$ such that
$
\Omega=((B\cap E)-x_0)\cap E,
$ as before. 

Notice that we can replace $\nabla\phi$ with a representative such that $\nabla\phi(x)=x+x_0$ for every $x\in \Omega$. Moreover, since  $\nabla\phi(\Omega)=B\cap E$, we also have that 
$
\Omega=B\cap E-x_0.
$
Thus, by the fact that $B\cap E-x_0=\Omega\subset E$, we have $-x_0\in \overline E$ and  
 relation \eqref{eq:proportionality} implies
\begin{equation}\label{eq:proportionality-1}
\omega(x+x_0)=\omega(x), \quad \forall x\in B\cap E-x_0.
\end{equation}
Assume by contradiction that $-x_0\in E$. Then, for every $t\in (0, 1/|x_0|)$, one has that $-tx_0\in B\cap E$. Applying \eqref{eq:proportionality-1} to $x=-(t+1)x_0\in  B\cap E-x_0$ yields
$
t^{\tau}\omega(-x_0)=(t+1)^\tau\omega(-x_0),
$
a contradiction. Hence $-x_0\in \partial E$. Since $\nabla\phi(B\cap E -x_0)=B\cap E$, we also have $x_0\in \overline E$. Since $-x_0\notin E$ and $E$ is convex, it cannot be $x_0\in E$. In particular $x_0\in \partial E \cap (-\partial E)$, as required.

Relation \eqref{u-lambda-components} implies that  
\[
|u|(x)=\tilde C_1 \mathbbm 1_{B\cap E} (x+x_0)\quad  \text{for a.e. $x\in B\cap E-x_0$},
\]
thus a change of variables and a scaling argument provide the general form of the extremal functions \eqref{Indicator}. In fact, it is enough to observe that the equality in inequality \eqref{log-Sobolev-1} holds for $u$ if and only if it holds for $x\mapsto \lambda^{n+\tau}u(\lambda x)$, $\lambda>0$. \hfill $\square$


\subsection{Proof of the inequality  \eqref{sharp-log-Sobolev-1}: case $p=1$, $\tau=0$} Since $\omega$ is  constant in $E$, see 
 Proposition \ref{prop-log-concave}, the proof from \S \ref{subsection-4-1} can be easily adapted to the present setting.  \hfill $\square$

\subsection{Equality in \eqref{sharp-log-Sobolev-1}: case $p=1$, $\tau=0$} The proof works exactly as in \S \ref{subsection-4-2} with some minor changes. 
In this case, the bijectivity of $\nabla \phi$ only gives $x_0\in (-\overline E)\cap \overline E$.   
Alternatively, we may obtain the same result by  Figalli,  Maggi and  Pratelli \cite{FMP2}. \hfill $\square$

%
%

\section{Application: sharp weighted  hypercontractivity}\label{section-5}

We recall that for a given function $g\colon E\to \mathbb R$, the Hopf-Lax formula has the expression
\begin{equation}\label{inf-convolution}
	{\bf Q}_{t}g(x)\coloneqq {\bf Q}_{t}^pg(x)=\inf_{y\in E}\left\{g(y)+\frac{|y-x|^{p'}}{p't^{p'-1}}\right\},
\end{equation}
where  $E\subseteq \mathbb R^n$ is an open convex cone and $p>1$;  we assume that $t>0$ is fixed such that ${\bf Q}_{t}g(x)>-\infty$ for all $x$ in $E$.  

We recall that in the classical literature (see e.g.\ Evans \cite{Evans}) it is assumed that $g$ is bounded  Lipschitz and as a consequence one obtains that the function $(x,t) \to {\bf Q}_t(g)(x)$ is well-defined for all $(x,t) \in E\times (0, \infty)$. Moreover, this function is locally Lipschitz and for a.e.   $(x,t)\in E\times (0,\infty)$ the Hamilton-Jacobi equation holds:  
		\begin{equation}\label{H-J}
			\frac{\partial}{\partial t}{\bf Q}_{t}g(x)+\frac{|\nabla {\bf Q}_{t}g(x)|^p}{p}=0.
		\end{equation}
		
		For our purposes the space of bounded Lipschitz functions does not work well. Indeed, observe that if $g$ is bounded and Lipschitz, then $\|e^g\|_{L^\alpha(\omega;E)}= +\infty$ and thus relation \eqref{hyperc-estimate} becomes trivial. 
		
		In order to handle this issue we shall consider the family of functions already considered in the Introduction of the paper: for some $t_0>0$, let
$$\mathcal  F_{t_0}(E)\coloneqq  \left\{ g:E\to \mathbb R:\begin{array}{lll}
	g {\rm \ is\ measurable,\ bounded\ from\ above\ and}\  \\
	{\rm there\ exists}\ x_0\in E {\rm \ such \ that} \ {\bf Q}_{t_0}g(x_0)>-\infty
\end{array}\right\}.
$$
  Our first result shows that if $t \in (0, t_0)$ then  $(x,t)\mapsto {\bf Q}_{t}g(x)$ is well-defined, locally Lipschitz, satisfies the Hamilton-Jacobi equation almost everywhere and also has strong global integrability properties: 
\begin{proposition}\label{proposition-hyper}
	Let $p>1$, $t_0>0$, $E\subseteq \mathbb R^n$ be an open convex cone and  $g \in \mathcal  F_{t_0}(E)$. Then the following statements hold. 
	\begin{itemize}
		\item[(i)] 	For every $(x,t)\in E\times (0,t_0)$ one has  ${\bf Q}_{t}g(x)>-\infty$ and for a.e.   $(x,t)\in E\times (0,t_0)$ the Hamilton-Jacobi equation \eqref{H-J} holds. 
	\item[(ii)] If  $\omega:E\to (0,\infty)$ is a homogeneous function and $0<\alpha_1\leq \alpha_2$ with $e^g\in L^{\alpha_1 p}(\omega;E)\cap L^{\alpha_2 p}(\omega;E)$ then for a.e. $t\in (0,t_0)$ the function $x\mapsto e^{q(t){\bf Q}_{t}g(x)}$ belongs to $W^{1,p}(\omega;E),$ where $q:[0,t_0]\to [\alpha_1,\alpha_2]$ is any  function.   
	\end{itemize}
\end{proposition}

\begin{proof}
	(i)  Let $K\subset E$ be any compact set and $0<t_1<t_2<t_0.$ We are going to prove that the function ${\bf Q}_{t}g(x)$, introduced in \eqref{inf-convolution}, is well-defined and Lipschitz continuous on $K\times [t_1,t_2]$. Since $t_2<t_0$ then for every $t\in [t_1,t_2]$, we have the coercivity property 
	$$\lim_{y\in E,|y|\to \infty}\left\{\frac{|y-x|^{p'}}{p't^{p'-1}}-\frac{|y-x_0|^{p'}}{p't_0^{p'-1}}\right\}=+\infty,$$
uniformly in $x\in K$. Consequently, the latter coercivity property  and the assumption ${\bf Q}_{t_0}g(x_0)>-\infty$ imply that for every $(x,t)\in K\times [t_1,t_2]$ one has
\begin{eqnarray*}
	{\bf Q}_{t}g(x)&=&\inf_{y\in E}\left\{g(y)+\frac{|y-x|^{p'}}{p't^{p'-1}}\right\}\\&\geq &\inf_{y\in E}\left\{g(y)+\frac{|y-x_0|^{p'}}{p't_0^{p'-1}}\right\}+\inf_{y\in E}\left\{\frac{|y-x|^{p'}}{p't^{p'-1}}-\frac{|y-x_0|^{p'}}{p't_0^{p'-1}}\right\}\\&=&
{\bf Q}_{t_0}g(x_0)+\inf_{y\in E}\left\{\frac{|y-x|^{p'}}{p't^{p'-1}}-\frac{|y-x_0|^{p'}}{p't_0^{p'-1}}\right\} >-\infty.
\end{eqnarray*}
In addition, we can find $R>0$ such that for every $t\in [t_1,t_2]$, we have $${\bf Q}_{t}g(x)=\inf_{y\in E\cap B(0,R)}\left\{g(y)+\frac{|y-x|^{p'}}{p't^{p'-1}}\right\}.$$
Note that the function $(x,t)\mapsto g(y)+\frac{|y-x|^{p'}}{p't^{p'-1}}$, $(x,t)\in K\times [t_1,t_2]$, is uniformly Lipschitz in $y$, thus $(x,t)\mapsto {\bf Q}_{t}g(x)$ is also  Lipschitz,  being the infimum of a family of Lipschitz functions. Thus, by Rademacher's theorem it follows that $(x,t)\mapsto {\bf Q}_{t}g(x)$   is differentiable a.e.\ in $K\times [t_1,t_2]$. In particular, by the arbitrariness of $K$ and $t_1,t_2\in (0,t_0)$, it turns out that $(x,t)\mapsto {\bf Q}_{t}g(x)$   is differentiable a.e.\ on $\mathbb R^n\times (0,t_0)$. Now, we can apply the proof of \cite[Theorem 5, p.\ 128]{Evans} to check that in the differentiable points of  $(x,t)\mapsto {\bf Q}_{t}g(x)$, the Hamilton-Jacobi equation \eqref{H-J} holds.  In fact, the quoted result is proved on the whole $\mathbb R^n$, for bounded Lipschitz initial datum $g$, but it can be easily adapted to $E\subseteq \mathbb R^n$ and $g \in \mathcal  F_{t_0}(E)$.

(ii) Let $0<\alpha_1\leq \alpha_2$  be two numbers with $e^g\in L^{\alpha_1 p}(\omega;E)\cap L^{\alpha_2 p}(\omega;E)$  and any function  $q:[0,t_0]\to [\alpha_1,\alpha_2]$. We first claim that $e^{q(t){\bf Q}_{t}g}\in L^p(\omega;E)$ for every $t\in (0,t_0)$. To do this, we  observe by definition that 
${\bf Q}_{t}g\leq g$ for every $t\in (0,t_0)$. Consequently, since $q(t)\in [\alpha_1,\alpha_2]$ for every $t\in [0,t_0]$, we have for every $t\in (0,t_0)$ and $x\in E$ that 
\begin{eqnarray*}
	0\leq e^{pq(t){\bf Q}_{t}g(x)}&\leq& e^{p\max\{\alpha_1 {\bf Q}_{t}g(x),\alpha_2 {\bf Q}_{t}g(x)\}}=\max\{e^{p\alpha_1 {\bf Q}_{t}g(x)},e^{p\alpha_2 {\bf Q}_{t}g(x)}\}\\&\leq& \max\{e^{p\alpha_1 g(x)},e^{p\alpha_2 g(x)}\}.
\end{eqnarray*}
Since by assumption one has $e^g\in L^{\alpha_1 p}(\omega;E)\cap L^{\alpha_2 p}(\omega;E)$,  the claim directly follows by the latter estimate. 

Now, for every $t\in (0,t_0)$ we consider the functions 
$$H_i(t)=\int_Ee^{p\alpha_i{\bf Q}_{t}g(x)}\omega(x)dx,\ \ i=1,2.$$
By the previous step, $H_i$ is well-defined on $ (0,t_0)$; moreover, since $t\mapsto {\bf Q}_{t}g(x)$ is non-increasing for every $x\in E$, the same holds for $H_i$ as well, $i=1,2.$ In particular, $H_i$ is differentiable a.e. on $(0,t_0)$ and $-\infty<\frac{d}{dt}H_i(t)<\infty$ for a.e.\ $t\in (0,t_0)$. In addition, by the Hamilton-Jacobi equation \eqref{H-J} if follows that for $i=1,2$ and for a.e. $t\in (0,t_0)$  we have
$$-\infty<\frac{d}{dt}H_i(t)=p\alpha_i\int_E\frac{\partial}{\partial t}{\bf Q}_{t}g(x) e^{p\alpha_i{\bf Q}_{t}g(x)}\omega(x)dx=-\alpha_i\int_E|\nabla {\bf Q}_{t}g(x)|^p e^{p\alpha_i{\bf Q}_{t}g(x)}\omega(x)dx,$$
i.e., 
$$\int_E|\nabla {\bf Q}_{t}g(x)|^p e^{p\alpha_i{\bf Q}_{t}g(x)}\omega(x)dx<+\infty,\ \ i=1,2,\ {\rm for \ a.e.}\ t\in (0,t_0).$$
By trivial interpolation, since $q(t)\in [\alpha_1,\alpha_2]$ for every $t\in [0,t_0]$, we obtain that
$$\int_E|\nabla {\bf Q}_{t}g(x)|^p e^{pq(t){\bf Q}_{t}g(x)}\omega(x)dx<+\infty\ \ {\rm for \ a.e.}\ t\in (0,t_0),$$
which is equivalent to 
$$\int_E|\nabla e^{q(t){\bf Q}_{t}g(x)}|^p \omega(x)dx<+\infty\ \ {\rm for \ a.e.}\ t\in (0,t_0).$$
In particular, the fact that $e^{q(t){\bf Q}_{t}g}\in L^p(\omega;E)$ combined with the last estimate implies that  $ e^{q(t){\bf Q}_{t}g}\in W^{1,p}(\omega;E)$ for a.e. $t\in (0,t_0)$. 
\end{proof}
Let us observe that if $g\colon E \to \mathbb R$ is bounded from above, then we can consider 
$$T = T(g) \coloneqq  \sup\{ t_0: g \in \mathcal  F_{t_0}(E) \}.$$
Clearly, the statement of the above proposition holds for the number $T$ replacing $t_0$.
On the other hand,  if $t >T$ then it also  follows from the statement of the proposition that ${\bf Q}_t(g)(x) = - \infty$ for all $x\in E$, and thus the statement of Theorem \ref{theorem-main-3} becomes trivial for the values $\beta >T$.

After this preparation, we are ready to prove Theorem \ref{theorem-main-3}. 

\subsection{Proof of the inequality \eqref{hyperc-estimate}} We first assume that $\alpha<\beta.$
 We  claim that 
 $g\in \mathcal F_{t_0}(E)$ with $e^g\in  L^\alpha(\omega;E)$ implies that $e^g\in L^{\gamma}(\omega;E)$ for every $\gamma\geq \alpha.$ Indeed, let $M>0$ be such that $g(x)\leq M$ for every $x\in E$. In addition, let $S=\{x\in E:g(x)\geq 0\}$. In particular, we have that $$\mu_\omega(S)\coloneqq \int_S\omega(x)dx\leq \int_Se^{\alpha g(x)}\omega(x)dx\leq \|e^g\|_{L^\alpha(\omega;E)}^\alpha<+\infty.$$
Moreover, since $\gamma\geq \alpha$, one has that
\begin{eqnarray*}
	\|e^g\|_{L^\gamma(\omega;E)}^\gamma&=&\int_Se^{\gamma g(x)}\omega(x)dx+\int_{E\setminus S}e^{\gamma g(x)}\omega(x)dx\leq e^{\gamma M}\mu_\omega(S)+\int_{E\setminus S}e^{\alpha g(x)}\omega(x)dx\\&\leq & (e^{\gamma M}+1)\|e^g\|_{L^\alpha(\omega;E)}^\alpha<+\infty,
\end{eqnarray*}
which ends the proof of the claim. 

Consequently, since $p>1$ and $\beta >  \alpha$, we have that $e^g\in L^{\alpha p}(\omega;E)\cap L^{\beta p}(\omega;E)$; moreover, by Proposition \ref{proposition-hyper}/(ii),  
for every function $q:[0,t_0]\to [\alpha,\beta]$  one has that  
\begin{equation}\label{sobolev-inclusion}
	e^{q(t){\bf Q}_{t}g}\in W^{1,p}(\omega;E)\ \ {\rm for\   a.e.}\ t\in (0,t_0).
\end{equation}

Fix arbitrarily $\tilde t\in (0,t_0)$ and let  $q:[0,\tilde t]\to [\alpha,\beta]$ be the function 
\begin{equation}\label{q-function}
	q(t)=\frac{\alpha\beta}{(\alpha-\beta)t/\tilde t+\beta},\ \ t\in [0,\tilde t].
\end{equation}
Note that $q$ is increasing on $[0,\tilde t]$ and   $q(0)=\alpha$ and $q(\tilde t)=\beta$. Since   relation \eqref{sobolev-inclusion} is valid a.e. on $(0,\tilde t)$,   the function 
$$F(t)=\|e^{{\bf Q}_{t}g}\|_{L^{q(t)}(\omega;E)}>0,\ \ t\in [0,\tilde t],$$ 
 is well-defined for a.e.\ $t\in [0,\tilde t]$ and absolutely continuous on $[0,\tilde t]$, see the proof of Proposition \ref{proposition-hyper}/(i). 
A direct computation and the Hamilton-Jacobi equation \eqref{H-J}  imply that for a.e. $t\in [0,\tilde t]$ we have 
\begin{eqnarray}\label{derivative-of -F}
\nonumber	F'(t)&=&F(t)^{1-q(t)}\frac{q'(t)}{q^2(t)}\left(\mathcal E_{\omega,E}(	e^{q(t){\bf Q}_{t}g})-F(t)^{q(t)}\log(F(t))^{q(t)}\right.\\&&\left.\qquad\qquad\qquad\qquad -\frac{q^{2-p}(t)}{pq'(t)}\int_Ee^{q(t){\bf Q}_{t}g(x)}|\nabla (q(t){\bf Q}_{t}g(x))|^p\omega(x)dx\right),
\end{eqnarray}
where we used the notation for the entropy from relation \eqref{sharp-log-Sobolev}. In fact, due to  \eqref{sobolev-inclusion}, we may apply the log-Sobolev inequality \eqref{sharp-log-Sobolev} for the normalized function	$u\coloneqq \frac{e^{\frac{q(t)}{p}{\bf Q}_{t}g}}{F(t)^\frac{q(t)}{p}}\in W^{1,p}(\omega;E)$ for a.e. $ t\in (0,\tilde t)$, obtaining that 
\begin{equation}\label{log-1}
	\frac{\mathcal E_{\omega,E}(	e^{q(t){\bf Q}_{t}g})}{F(t)^{q(t)}}-\log(F(t))^{q(t)}\leq \frac{n+\tau}{p}	\log\left(\frac{\mathcal L_{\omega,p}}{p^p}\frac{\displaystyle\int_Ee^{q(t){\bf Q}_{t}g(x)}|\nabla (q(t){\bf Q}_{t}g(x))|^p\omega(x)dx}{F(t)^{q(t)}}\right).
\end{equation}
By using the elementary inequality $\log(ey)\leq y$ for every $y>0$ (with equality only for $y=1$), the right hand side \textit{RHS} of the above inequality can be estimated for every $s>0$ by  
\begin{equation}\label{log-2}
	RHS\leq\frac{n+\tau}{p} \log\left(s\frac{\mathcal L_{\omega,p}}{ep^p}\right)+\frac{n+\tau}{ps}\frac{\displaystyle\int_Ee^{q(t){\bf Q}_{t}g(x)}|\nabla (q(t){\bf Q}_{t}g(x))|^p\omega(x)dx}{F(t)^{q(t)}}.
\end{equation}
 Combining relations \eqref{derivative-of -F}-\eqref{log-2} with the choice 
$$s\coloneqq \frac{(n+\tau)q'(t)}{q^{2-p}(t)}>0,$$
it turns out that 
\begin{equation}\label{F-function}
	\frac{F'(t)}{F(t)}\leq \frac{n+\tau}{p}\frac{q'(t)}{q^2(t)} \log\left(\frac{\mathcal L_{\omega,p}}{ep^p}\frac{(n+\tau)q'(t)}{q^{2-p}(t)}\right)\ \ {\rm for\   a.e.}\ t\in [0,\tilde t].
\end{equation}
After an integration of the above inequality on $[0,\tilde t]$, we obtain that 
$$\log\frac{F(\tilde t)}{F(0)}\leq \frac{n+\tau}{p}\int_0^{\tilde t}\frac{q'(t)}{q^2(t)} \log\left(\frac{\mathcal L_{\omega,p}}{ep^p}\frac{(n+\tau)q'(t)}{q^{2-p}(t)}\right)dt.$$
The choice of the function $q$  is confirmed at this stage, since the minimum of the right hand side of the above estimate is realized by $q$ from  \eqref{q-function}, solving the corresponding  Euler-Lagrange equation with the constraints $q(0)=\alpha$ and $q(\tilde t)=\beta$.

   Since $F(\tilde t)=\|e^{{\bf Q}_{\tilde t}g}\|_{L^{\beta}(\omega;E)}$ and $F(0)=\|e^{g}\|_{L^{\alpha}(\omega;E)}$, a direct computation of the above integral yields 
$$\|e^{{\bf Q}_{\tilde t}g}\|_{L^{\beta}(\omega;E)}\leq \|e^{g}\|_{L^{\alpha}(\omega;E)}\left(\frac{\beta-\alpha}{\tilde t}\right)^{\frac{n+\tau}{p}\frac{\beta-\alpha}{\alpha\beta}}\frac{\alpha^{\frac{n+\tau}{\alpha\beta}(\frac{\alpha}{p}+\frac{\beta}{p'})}}{\beta^{\frac{n+\tau}{\alpha\beta}(\frac{\beta}{p}+\frac{\alpha}{p'})}}\left((p')^\frac{n+\tau}{p'}\Gamma\left(\frac{n+\tau}{p'}+1\right)\int_{B \cap E}\omega\right)^{\frac{\alpha-\beta}{\alpha\beta}},$$
which is precisely the inequality \eqref{hyperc-estimate} whenever $\alpha<\beta$. 

When $\alpha=\beta$,  the inequality \eqref{hyperc-estimate} reduces to $\|e^{{\bf Q}_{\tilde t}g}\|_{L^{\alpha}(\omega;E)}\leq \|e^{g}\|_{L^{\alpha}(\omega;E)},$ which directly follows by  ${\bf Q}_{\tilde t}g\leq g.$

\subsection{Equality in \eqref{hyperc-estimate}}
We first assume again that $\alpha<\beta$. If the function $g:E\to \mathbb R$ has the form from \eqref{extremal-hyperc}, after a direct computation we observe that equality holds in \eqref{hyperc-estimate}; a similar proof can be found in Gentil \cite{Gentil} for the unweighted case and $E=\mathbb R^n$. 

Conversely, assume that $\alpha<\beta$ and  there exits equality in \eqref{hyperc-estimate} for some $\tilde t \in (0,t_0)$ and $g\in \mathcal F_{\tilde t}(E)$ with $e^g\in  L^\alpha(\omega;E)$. In particular, by tracking back the inequalities in the proof, it turns out that we should have equalities both in \eqref{log-1} and \eqref{log-2}; namely,  we have that
\begin{equation}\label{first-rel}
	\frac{e^{\frac{q(t)}{p}{\bf Q}_{t}g(x)}}{F(t)^\frac{q(t)}{p}}=u_{\lambda_t,x_0^t}(x) \ \ {\rm for \ every}\ x\in E
\ {\rm and\ a.e.}\ t\in [0,\tilde t],\end{equation}
for some $\lambda_t>0$  and $x_0^t$ such that 
$x_0^t\in -\partial E\cap \partial E$ and $\omega(x+x_0^t)=\omega(x)$ for every $x\in E$ whenever $\tau>0$, and 
 $x_0^t\in -\overline E\cap \overline E$ and $\omega$ is constant in $E$ whenever $\tau=0$, see \eqref{Gaussian}, and
\begin{equation}\label{second-rel}
\frac{\displaystyle\int_Ee^{q(t){\bf Q}_{t}g(x)}|\nabla (q(t){\bf Q}_{t}g(x))|^p\omega(x)dx}{F(t)^{q(t)}}=\frac{(n+\tau)q'(t)}{q^{2-p}(t)}\ \ {\rm for\ a.e.}\ t\in [0,\tilde t].
\end{equation}
By replacing \eqref{first-rel} into \eqref{second-rel}, we obtain that
$$p^p\int_{E}|\nabla u_{\lambda_t,x_0^t}(x)|^p\omega(x)dx=\frac{(n+\tau)q'(t)}{q^{2-p}(t)}\ \ {\rm for\ a.e.}\ t\in [0,\tilde t].$$
Relation \eqref{Gaussian},  the fact that $\omega(x+x_0^t)=\omega(x)$ for every $x\in E$ whenever $\tau>0$ (and $\omega$ is constant in $E$ whenever $\tau=0$) and a change of variables imply that  
$$\int_{E}|\nabla u_{\lambda_t,x_0^t}(x)|^p\omega(x)dx=\int_{E}|\nabla u_{\lambda_t,0}(x)|^p\omega(x)dx=\lambda_t^{p-1}\frac{n+\tau}{p'}\frac{(p')^{p}}{p^p}.$$
Therefore,  one has that 
$$\lambda_t=\frac{1}{p'}\left(\frac{\beta-\alpha}{\tilde t}\right)^\frac{1}{p-1}\frac{\alpha\beta}{\left((\alpha-\beta)t/\tilde t+\beta\right)^{p'}},\ t\in[0,\tilde t].$$

\noindent Note that relation \eqref{first-rel} can be equivalently transformed into 
\begin{eqnarray}\label{Q-t-explicit}
\nonumber	{\bf Q}_{t}g(x)&=&\log F(t)+\frac{p}{q(t)}\log u_{\lambda_t,x_0^t}(x)\\&=&\nonumber\log F(t)++\frac{p}{q(t)}\log\left(\lambda_t^\frac{n+\tau}{pp'}\left(\Gamma\left(\frac{n+\tau}{p'}+1\right)\int_{B\cap E}\omega\right)^{-\frac{1}{p}}\right)-\frac{\lambda_t}{q(t)}|x+x_0^t|^{p'}\\&=:&h(t)-\frac{b_t}{p'}{|x+x_0^t|^{p'}},
\end{eqnarray}
which is valid for every $t\in (0,\tilde t].$

On one hand, we notice  that we have equality also in 
\eqref{F-function}. Moreover, a direct computation yields that the equality in \eqref{F-function} is equivalent to the vanishing of the 
 derivative of  $t\mapsto h(t)$ on  $(0,\tilde t)$; in particular, for every $t\in (0,\tilde t]$ we have  
\begin{equation}\label{h-constant}
	h(t)=\lim_{t\to 0}h(t)=:C\in \mathbb R.
\end{equation}


On the other hand, by the definition of 	${\bf Q}_t g$, one has that
$$	{\bf Q}_{t}g(x)=\inf_{y\in E}\left\{\frac{|y-x|^{p'}}{p't^{p'-1}}-(-g)(y)\right\}, \ x\in E,\ t\in (0,\tilde t].$$
The latter relation gives the hint to use properties of $c$-concave functions associated with the cost function $c(x,y)\coloneqq  c_t(x,y)=\frac{|x-y|^{p'}}{p't^{p'-1}}.$ 

Let us recall from Villani \cite{Villani} that a function $\phi: E\to \mathbb R$ is $c$-concave if it can be written as 
$$ \phi(x) =\inf_{y\in E} \{c(x,y) - \psi(y)\} , x\in E,$$
for some function $\psi:E \to \mathbb R$ that is not identically $- \infty$. We can denote the right side of the above expression by $\psi^c$ and call it the $c$-transform of $\psi$. In this way, a function is $c$-concave if it can be written as $\phi = \psi^c$ for some other function $\psi$. It can be easily verified just by the definition that $\phi \leq (\phi^c)^c$ and $\phi$ is $c$-concave if and only if $\phi = (\phi^c)^c$. 

Let us assume momentarily that $\phi = -g$ is $c_{\tilde t}$-concave. It follows by \cite[Proposition 5.47]{Villani} that $-g$ will be $c_t$-concave for all $0< t \leq \tilde t$.  Using that $$-g = ((-g)^{c_t}))^{c_t}$$  and writing the $c_t$-transform in terms of ${\bf Q}_t$ we obtain the relation
$$-g(y)=\inf_{x\in E}\left\{\frac{|y-x|^{p'}}{p't^{p'-1}}-	{\bf Q}_{t}g(x)\right\}={\bf Q}_t (-{\bf Q}_t g)(y),\ y\in E.$$ 
Therefore, by \eqref{Q-t-explicit} and \eqref{h-constant}, we have that 
\begin{equation}\label{g-explicit}
	-g(y)=-C+	{\bf Q}_t\left(\frac{b_t}{p'}|\cdot+x_0^t|^{p'}\right)(y),\ y\in E.
\end{equation}
A direct computation gives that for every $b>0$ and $x_0\in E$, one has	
\begin{equation}\label{Q-t-explicit-0}
	{\bf Q}_t\left(\frac{b}{p'}|\cdot+x_0^t|^{p'}\right)(y)=\frac{b}{p'\left(1+tb^{{p-1}}\right)^{p'-1}}|y+x_0^t|^{p'},\ y\in E;
\end{equation}
applying the latter formula for \eqref{g-explicit} and recalling that $b_t=\frac{p'\lambda_t}{q(t)}>0$, it turns out that 
\begin{equation}\label{function-g}
	g(y)=C-\frac{1}{p'}\left(\frac{\beta-\alpha}{\beta\tilde t}\right)^\frac{1}{p-1}|y+x_0^t|^{p'},\ y\in E.
\end{equation}
Since $g$ is independent on $t>0$, we necessarily have that $x_0\coloneqq x_0^t$ for every $t\in (0,\tilde t];$ this concludes the proof of the relation \eqref{extremal-hyperc} whenever $-g$ is  $c_{\tilde t}$-concave. A straightforward computation shows that $e^g\in  L^\alpha(\omega;E)$.

We also observe that $g\in \mathcal F_{\tilde t}(E)$. To see this,  by a direct computation  we have that 
\begin{equation}\label{Q-t}
	{\bf Q}_t\left(-\frac{b}{p'}|\cdot+x_0|^{p'}\right)(x)=-\frac{b}{p'\left(1-tb^{{p-1}}\right)^{p'-1}}|x+x_0|^{p'},\ x\in E,
\end{equation}
for every $x_0\in E$ and $0<tb^{p-1}<1.$ Clearly, $g$ is bounded from above, and as for $$b_0\coloneqq \left(\frac{\beta-\alpha}{\beta\tilde t}\right)^\frac{1}{p-1}$$ we have that $\tilde tb_0^{p-1}<1,$ it yields that 
$g\in \mathcal F_{\tilde t}(E)$, see  \eqref{Q-t}.

Finally, we shall remove the assumption that $-g$ is $c_{\tilde t}$-concave. Assume by contradiction that there exists another extremal function $\tilde g$ in \eqref{hyperc-estimate}  such that $-\tilde g$ is not $c_{\tilde t}$-concave.  By tracking back again the equality cases like in the first part of the proof, see \eqref{Q-t-explicit}, we obtain also for this function $\tilde g$  that 
\begin{equation}\label{Q-t-tilde}
	{\bf Q}_{t}\tilde{g}(x) = C-\frac{b_t}{p'}{|x+x_0|^{p'}},\ \ x\in E,\ t\in (0,\tilde t].
\end{equation}
Moreover,  
 if $g$ is the $c_{\tilde t}$-concave function from \eqref{function-g}, it turns out by relations \eqref{Q-t-tilde}  and \eqref{Q-t-explicit-0}  that
 $$-\tilde g(y)\leq {\bf Q}_{t}(-{\bf Q}_{t}\tilde{g})(y)=-g(y),\ y\in E,\ t\in (0,\tilde t],$$ 
 thus $\tilde g\geq g$ on $E.$  
 On the other hand, both functions $g$ and $\tilde g $ are  extremals in the hypercontractivity inequality \eqref{hyperc-estimate}, i.e., in particular, we have $$\frac{\|e^{{\bf Q}_{\tilde t}\tilde g}\|_{L^{\beta}(\omega;E)}}{\|e^{\tilde g}\|_{L^{\alpha}(\omega;E)}} =\frac{\|e^{{\bf Q}_{\tilde t}g}\|_{L^{\beta}(\omega;E)}}{\|e^{g}\|_{L^{\alpha}(\omega;E)}}.$$
 Furthermore, since ${\bf Q}_{\tilde t} g = {\bf Q}_{\tilde t} \tilde g$, up to a translation (see \eqref{Q-t-explicit} and \eqref{Q-t-tilde}), it follows that $\|e^{{\bf Q}_{\tilde t}\tilde g}\|_{L^{\beta}(\omega;E)} = \|e^{{\bf Q}_{\tilde t}g}\|_{L^{\beta}(\omega;E)},$ which implies $$\|e^{g}\|_{L^{\alpha}(\omega;E)} = \|e^{\tilde g}\|_{L^{\alpha}(\omega;E)}.$$ This relation together with $\tilde g\geq g$ on $E$ implies  that $ \tilde g =g$ a.e.\ on $E$, which is a  contradiction. 
 
 It remains to analyze the equality in  \eqref{hyperc-estimate} whenever $\alpha=\beta$; namely, we assume that equality holds in  \eqref{hyperc-estimate} for some $\tilde t \in (0,t_0)$ and $g\in \mathcal F_{\tilde t}(E)$ with $e^g\in  L^\alpha(\omega;E),$ i.e.,  $\|e^{{\bf Q}_{\tilde t}\tilde g}\|_{L^{\alpha}(\omega;E)} = \|e^{g}\|_{L^{\alpha}(\omega;E)}.$ Since we have that ${\bf Q}_{\tilde t}g\leq g$ on $E$, the latter equality implies that ${\bf Q}_{\tilde t}g= g$. 
 Moreover,  the monotonicity of $t\mapsto {\bf Q}_tg$ implies that for every $t\in [0,\tilde t]$, 
 $${\bf Q}_{\tilde t}g\leq {\bf Q}_{ t}g\leq g={\bf Q}_{\tilde t}g,$$
 thus ${\bf Q}_{ t}g=g$ for every $t\in [0,\tilde t]$. In particular, $ 	\frac{\partial}{\partial t}{\bf Q}_{t}g=0$ and by the Hamilton-Jacobi equation \eqref{H-J} it follows that ${|\nabla {\bf Q}_{t}g(x)|^p}=0$ for every $t\in (0,\tilde t)$ and a.e. $x\in E.$ Since ${\bf Q}_tg$ is locally Lipschitz, it turns out that $g={\bf Q}_{ t}g\equiv C$ for some $C\in\mathbb R,$ which contradicts the fact that $e^g\in  L^\alpha(\omega;E)$. This concludes the proof. 
\hfill $\square$




\begin{remark}\rm 
As a final remark we should mention that there is a well-known equivalence between the log-Sobolev inequality, hypercontractivity of the Hopf-Lax semigroup and Pr\'ekopa-Leindler inequalities, see \cite{Gentil}, \cite{delPinoDolbeaultGentil}. It would be therefore feasible that this equivalence is also reflected in the  characterization of the equality cases in these inequalities. Consequently, one possible  route to obtain a characterization of the extremals in the weighted log-Sobolev inequality could have been done as follows:  Start from weighted  Pr\'ekopa-Leindler inequalities as shown by  Milman and Rotem \cite{MR_Adv}. Next, obtain a characterization of the equality cases here using the techniques of Balogh and Krist\'aly \cite{BK_Adv} or Dubuc \cite{Dubuc}. Using the characterization of the equality cases in the weighted Pr\'ekopa-Leindler inequality, the characterization of the  equality case in the Hopf-Lax hypercontractivity estimate could be deduced and could be used for a  characterization of the extremals in the weighted log-Sobolev inequality. We think that this strategy could be worked out, however the technical complexity involved in this process would be at least comparable if not higher to the one of the present paper.  
 \end{remark}




\noindent {\bf Acknowledgment.} 
The authors thank the anonymous Referee for all valuable comments and suggestions that significantly improved the presentation of the paper.
A. Krist\'aly also thanks the members of Institute of Mathematics,  Universit\"at Bern, for their invitation and kind hospitality, where the present work has been initiated.

 

\begin{thebibliography}{99}
%
%

\bibitem{AGK} M. Agueh, N. Ghoussoub, X. Kang,
Geometric inequalities via a general comparison principle for interacting gases. 
\textit{Geom. Funct. Anal.} 14 (2004), no. 1, 215--244. 

\bibitem{ABS} L. Ambrosio, E. Brué, D. Semola, Lectures on optimal transport. Unitext, 130. La Matematica per il 3+2. Springer, Cham, 2021.

		
	
	
	
 	
 		\bibitem{BGK_PLMS} Z.M. \ Balogh, C.\ Guti\'{e}rrez,  A.\ Krist\'{a}ly, Sobolev inequalities with jointly concave weights on convex cones. \textit{Proc. Lond. Math. Soc.} 122(4) (2021), 537--568.
 		
 		\bibitem{BK_Adv} Z.M. \ Balogh, A. \ Krist\'aly, Equality in Borell-Brascamp-Lieb inequalities on curved spaces. \textit{Adv. Math.} 339 (2018), 453--494. 
 		
 			\bibitem{BartheK}  F. Barthe, A.V. Kolesnikov,  Mass transport and variants of the logarithmic Sobolev inequality. \textit{J. Geom. Anal.} 18 (2008), no. 4, 921--979.
 		
 		\bibitem{Beckner} W. Beckner, Geometric asymptotics and the logarithmic Sobolev inequality. \textit{Forum Math.} 11 (1999), 105--137. 
 		
 	\bibitem{Beckner-Pearson}	W. Beckner, M. Pearson, On sharp Sobolev embedding and the logarithmic Sobolev inequality. \textit{Bull. London Math. Soc.} 30 (1998), 80--84.
 		
 		
%

	\bibitem{BBF} G.\ Bellettini, G.\ Bouchitt\'e, I.\ Fragal\`a, BV functions with respect to a measure and relaxation of metric integral functionals. \textit{J. Convex Anal.}, 6(2) (1999), 349--366.

 	
 	\bibitem{Brenier} Y. Brenier, Polar factorization and
 	monotone rearrangement of vector-valued functions.\textit{ Comm. Pure Appl. Math.}  44 (1991), no. 4, 375--417.
 	
 	
 	\bibitem{Cabre-Ros-Oton}  X. Cabr\'e, X. Ros-Oton,  Sobolev and isoperimetric inequalities with monomial weights. \textit{J. Differential Equations} 255 (2013), no. 11, 4312--4336. 
 	
 	\bibitem{Cabre-Ros-Oton-Serra}  X. Cabr\'e, X. Ros-Oton, J. Serra, Sharp isoperimetric inequalities via the ABP method. \textit{J. Eur. Math. Soc.} 18 (2016),  2971--2998. 
 
 \bibitem{CGPROS} 	E. \ Cinti, F. \ Glaudo, A. \ Pratelli, X. \ Ros-Oton,  J. Serra, Sharp quantitative stability for isoperimtric inequalities with homogeneous weights. \textit{ Tran. Amer. Math. Soc.}, to appear. Link: https://doi.org/10.1090/tran/8525
 

 
 	
 	
		\bibitem{CFR} G. Ciraolo, A. Figalli, A. Roncoroni,	Symmetry results for critical anisotropic $p$-Laplacian equations in convex cones. \textit{Geom. Funct. Anal.} 30 (2020), no. 3, 770--803. 
	
	%
	
	
	 \bibitem{Cordero} D. \ Cordero-Erausquin, Some applications of mass transport to Gaussian-type inequalities. \textit{Arch. Ration. Mech. Anal.} 161 (2002), no. 3, 257--269.
		\bibitem{CENV} D. Cordero-Erausquin, B. Nazaret, C. Villani, A mass-transportation approach to sharp Sobolev and Gagliardo-Nirenberg inequalities. \textit{Adv. Math.} 	182, 2, (2004), 307--332. 
	
	%
	\bibitem{delPinoDolbeault-2} M. Del Pino, J.\ Dolbeault, The optimal Euclidean $L^p$-Sobolev logarithmic inequality. \textit{J. Funct. Anal.} 197 (2003), no. 1, 151--161.

\bibitem{delPinoDolbeaultGentil} M. Del Pino, J.\ Dolbeault, I. \ Gentil, Nonlinear diffusions, hypercontractivity and the optimal $L^p$- Euclidean logarithmic Sobolev inequality. \textit{J. Math. Anal. Appl.}, 293 (2004), 375--388.


\bibitem{Dubuc} S. Dubuc, Crit\`eres de convexit\'e et in\'egalit\'es int\`egrales. \textit{Ann. Inst. Fourier} (Grenoble) 27 (1) (1977) 135--165. 

\bibitem{Evans} L.C. \ Evans, 
\textit{Partial differential equations}. 
Graduate Studies in Mathematics, 19. American Mathematical Society, Providence, RI, 1998.
	
\bibitem{EG} L.C.\ Evans, R.F. \ Gariepy, \textit{Measure theory and fine properties of functions}. \textit{Studies in Advanced Mathematics} CRC Press, (1992). 
	
	\bibitem{Fathi} M. Fathi, E. Indrei, M. Ledoux, Quantitative logarithmic Sobolev inequalities and stability estimates. \textit{Discrete Contin. Dyn. Syst.} 36 (2016), no. 12, 6835--6853. 
	
	
	\bibitem{Gentil} I.\ Gentil,  The general optimal $L^p$-Euclidean logarithmic Sobolev inequality by Hamilton-Jacobi equations. \textit{J. Funct. Anal.} 202(2) (2003), 591--599.  
	
	
	\bibitem{Gross}  L. Gross,  Logarithmic Sobolev inequalities. \textit{Amer. J. Math.} 97 (1975), no. 4, 1061--1083.
	
	
	%
	%

	
	%
	%
	%
	%
	
	%
	%
	%

	%

\bibitem{FMP} A.\ Figalli, F.\ Maggi, A.\ Pratelli,  A mass transportation approach to quantitative isoperimetric
inequalities. \textit{Invent. Math.} 182(1) (2010), 167--211.

\bibitem{FMP2} A.\ Figalli, F.\ Maggi, A.\ Pratelli, Sharp stability theorems for the anisotropic Sobolev and log-Sobolev inequalities on functions of bounded variation. \textit{Adv. Math.} 242, (2013), 80--101.

\bibitem{Fujita} Y. Fujita, A supplementary proof of $L^{p}$-logarithmic Sobolev inequality. 
\textit{Ann. Fac. Sci. Toulouse Math.} (6) 24 (2015), no. 1, 119--132. 


	%
	%
	%

	%
	%
	%
\bibitem{Lam} N.\ Lam, Sharp weighted isoperimetric and Caffarelli-Kohn-Nirenberg inequalities.
	\textit{Adv. Calc. Var.} 14 (2021), no. 2, 153--169.
	
\bibitem{Ledoux} M. Ledoux, Isoperimetry and Gaussian analysis, in: Lectures on Probability Theory and Statistiques, Ecole d'\'et\'e de probabilit\'es de St-Flour, 1994, in: Lecture Notes in Math., vol. 1648, Springer, Berlin, 1996, pp. 165--294.
	
	%
	%
	%
	%
	%
	%
	%
	%
	%
	%
	%
	\bibitem{Maggi} F. Maggi,  Optimal mass transport on Euclidean spaces. Cambridge Studies in Advanced Mathematics, 207. Cambridge University Press, Cambridge, 2023. 
	
	\bibitem{MR_Adv}  E. \ Milman, L. \ Rotem, Complemented Brunn-Minkowski inequalities and isoperimetry for homogeneous and non-homogeneous measures. \textit{Adv. Math.} 262 (2014), 867--908.
	
	\bibitem{McCann} R.J. McCann, Existence and uniqueness of monotone measure-preserving maps.
	\textit{Duke Math. J.} 80 (1995), no. 2, 309--323. 
	
	
		\bibitem{McCann-Adv} R.J. McCann,
	A convexity principle for interacting gases. 
	\textit{Adv. Math.} 128 (1997), no. 1, 153--179. 
	
	\bibitem{Nguyen} V.H.\ Nguyen, Sharp weighted Sobolev and Gagliardo-Nirenberg inequalities on half-spaces via mass transport and consequences. \textit{Proc. Lond. Math. Soc.} (3) 111 (2015), no. 1, 127--148.
	%
	%
	%
	%
	%
	%
	\bibitem{Villani} C. Villani, \textit{Topics in Optimal Transportation}, vol. 58 of \textit{Graduate Studies in Mathematics.} Amer. Math. Soc. Providence, RI, 2003. 	
	
	\bibitem{Weissler}  F.B. Weissler,  Logarithmic Sobolev inequalities for the heat-diffusion semigroup. \textit{Trans. Amer. Math. Soc.} 237 (1978), 255--269.
	
\bibitem{Willem}	M. Willem, 
	\textit{Minimax theorems}.	Progress in Nonlinear Differential Equations and their Applications, 24. Birkh\"auser Boston, Inc., Boston, MA, 1996.
\end{thebibliography}
\end{document}